\documentclass[pdflatex,10pt]{article}
\usepackage[utf8]{inputenc}

\usepackage{amsmath,color}
\usepackage{amssymb,amsthm}
\usepackage{subfig}
\usepackage[permil]{overpic}
\usepackage[multiple]{footmisc}
\usepackage{xr}
\usepackage[left=1.6cm,right=1.6cm,top=2.50cm,bottom=2.50cm]{geometry}
\usepackage{graphicx}
\usepackage[font=small,labelfont=bf,
   justification=justified,
   format=plain,labelsep=space]{caption}
\usepackage{float}
\usepackage[greek,english]{babel}
\usepackage{teubner}
\usepackage[font=small,labelfont=bf,justification=centering]{caption}
\usepackage[T1]{fontenc}
\usepackage{lastpage}
\usepackage{enumerate}
\usepackage{enumitem}
\usepackage{lmodern}
\usepackage{array}
\usepackage{bm}
\usepackage{multicol}
\usepackage{multirow}
\usepackage{dsfont}
\usepackage{tensor}
\usepackage{fancyhdr}
\usepackage{listings}
\usepackage{dsfont}
\usepackage{siunitx}
\usepackage{titling}
\usepackage{lipsum}
\usepackage{tabularx}
\usepackage{verbatim}
\usepackage{authblk}
\usepackage{csquotes}
\usepackage{chemfig} 
\usepackage{bbm} 
\usepackage{hyperref} 
\usepackage{ctable} 
\usepackage{multirow} 
\usepackage{titlesec} 
\usepackage{marvosym}
\usepackage{relsize,exscale} 
\usepackage{setspace}
\usepackage{filecontents}

\allowdisplaybreaks

\graphicspath{{Fig/}}

\addto\captionsenglish{}

\titleformat{\subsection}{\normalfont\large\raggedright\it}{\thesubsection}{1em}{}

\makeatletter
\newcommand{\thickhline}{%
    \noalign {\ifnum 0=`}\fi \hrule height 1pt
    \futurelet \reserved@a \@xhline
}
\newcolumntype{"}{@{\hskip\tabcolsep\vrule width 1pt\hskip\tabcolsep}}
\makeatother

\usepackage{adjustbox}

\def\txte{{\textnormal{e}}}

\makeatletter
\newsavebox{\@brx}
\newcommand{\llangle}[1][]{\savebox{\@brx}{\(\m@th{#1\langle}\)}%
  \mathopen{\copy\@brx\kern-0.5\wd\@brx\usebox{\@brx}}}
\newcommand{\rrangle}[1][]{\savebox{\@brx}{\(\m@th{#1\rangle}\)}%
  \mathclose{\copy\@brx\kern-0.5\wd\@brx\usebox{\@brx}}}
\makeatother

\newcommand{\RomanNumeralCaps}[1]
    {\MakeUppercase{\romannumeral #1}}

\newtheorem{deff}{Definition}[section]

\newtheorem{thm}[deff]{Theorem}
\newtheorem{lm}[deff]{Lemma}
\newtheorem{cor}[deff]{Corollary}
\theoremstyle{definition}

\newtheorem{remark}[deff]{Remark}

\numberwithin{equation}{section}

\usepackage{graphicx} 

\title{Probability of Transition to Turbulence\\ in a Reduced Stochastic Model of Pipe Flow}
\author[1,$\dag$]{P. Bernuzzi}
\author[1,2]{C. Kuehn}
\affil[1]{\footnotesize{Technical University of Munich, School of Computation Information and Technology, Department of
Mathematics, Boltzmannstraße 3, 85748 Garching, Germany}}
\affil[2]{\footnotesize{Technical University of Munich, Munich Data Science Institute, Walther-von-Dyck-Straße 10, 85748 Garching, Germany}}
\affil[$\dag$]{Author to whom any correspondence should be addressed.\smallskip 

Email addresses: paolo.bernuzzi@ma.tum.de (P. Bernuzzi), ckuehn@ma.tum.de (C. Kuehn).}
\date{\today}

\begin{document}

\maketitle


\begin{abstract}
    We study the phenomenon of turbulence initiation in pipe flow under different noise structures by estimating the probability of initiating metastable transitions. We establish lower bounds on turbulence transition probabilities using linearized models with multiplicative noise near the laminar state. First, we consider the case of stochastic perturbations by It\^o white noise; then, through the Stratonovich interpretation, we extend the analysis to noise types such as white and red noise in time. 
    Our findings demonstrate the viability of detecting the onset of turbulence as rare events under diverse noise assumptions. The results also contribute to applied SPDE theory and offer valuable methodologies for understanding turbulence across application areas.
\end{abstract}


\small{\textbf{MSC codes:} 
37A50,   
60H15,   
76F06,  
76M35.   

\small{\textbf{Funding:} This work was supported by the European Union’s Horizon 2020 research and innovation programme under Grant Agreement 956170.}}

\pagestyle{fancy}
\fancyhead{}
\renewcommand{\headrulewidth}{0pt}
\fancyhead[C]{\textit{Probability of Transition to Turbulence in a Reduced Stochastic Model of Pipe Flow}}

\section{Introduction}

The flow of a fluid is known to display different behaviour depending on the assumptions on the fluid and the shape of the geometry traversed. We focus on plane Couette flow, or flow along a pipe, and the transition to turbulence starting from an initially laminar state \cite{landau1987fluid}. Recently, this area has been connected to directed percolation, a classical area of probability theory. It is known that the transition to turbulence crucially depends upon the Reynolds number, i.e., on the ratio between inertial and viscous forces in a fluid. From a mathematical perspective, studying the transition to turbulence directly in the Euler or Navier-Stokes equations turns out to be extremely difficult. Yet, quite recently, other simplified models have been proposed and directly validated against experiments.
\smallskip

The first model proposed in \cite{barkley2011modeling,barkley2016theoretical} is defined by two coupled stochastic partial differential equations (SPDEs) with second-order dissipation, an advection term, and multiplicative degenerate It\^o noise \cite{munoz2003multiplicative}. The solution of this system may display structures labeled as slugs for high Reynolds numbers, and the turbulent state covers the whole pipe in finite time with a high probability. Conversely, for low values of the corresponding parameter, the turbulent regions are absorbed in the laminar state in finite time with high probability. Lastly, traveling structures labeled as puffs indicate a fluctuating transition for intermediate values. In such cases, metastable transitions between turbulent and laminar states are observed in the pipe. The study of the rise and the splitting of puffs is relevant to describe the transition to turbulence as a rare event \cite{gome2022extreme}. The steps involved in such an occurrence can be described through the computation of instantons, most likely paths to display rare occurrences \cite{bernuzzi2024large,lestang2018computing,wang2020improvements}. 
\smallskip

The complex structure provided by the model discussed above is not observed for all fluids and is in contrast with the perspective of \cite{pomeau1986front,pomeau2015transition}. In this case, the turbulent state is seen as a fluctuating state able to decay spontaneously. Conversely, the laminar state is an absorbing state, which is unable to induce turbulence into the system. The alternative model proposed in \cite{gome2024phase} is a one-dimensional SPDE that does not display advection. Furthermore, this simplified model does not show well-defined traveling states such as puffs and slugs. For this reason, the flow is labeled band-free. In particular, the transition to turbulence is primarily driven by noise and occurs more rarely than in the previously discussed cases. The transition mechanism can be studied numerically through the adjoint state method \cite{bernuzzi2024large}, which captures the instanton for various types of noise. Alternatively, the Trajectory-Adaptive Multilevel Sampling algorithm (TAMS) \cite{lestang2018computing,wang2020improvements} computes different trajectories that display such an event. Along with the description of such an occurrence, the analytic estimation of the probability of jumps provides a valuable tool to predict the likelihood of this phenomenon. 
\smallskip

The metastable jump to turbulence is rare under the assumption of initial conditions close to the laminar state. Under such conditions, we label this transition as turbulence initiation, or turbulence onset, to indicate that it is primarily induced by noise. Therefore, the linearization of the model near such a steady state is a natural simplification. The linearized system resembles the cable equation \cite{walsh1986introduction}, or a heat equation with a second linear dissipative term, with multiplicative noise. The literature regarding such a model is vast. These types of SPDEs have been studied in a theoretical context to understand the influence of noise on potential finite-time blow-up of the solution. In particular, for standard white noise multiplied by a multiplicative term of order $\gamma> \frac{3}{2}$ the mild solution is proven to diverge to infinity in finite time \cite{mueller2000critical,mueller1993blowup}. Conversely, for a multiplicative term of order $0\leq \gamma\leq \frac{3}{2}$ the mild solution of the linear system does not explode in finite time \cite{mueller1991long,salins2023solutions}. Although the band-free plane Couette flow model considers noise of order $\gamma=1$, we show in this paper that analytic techniques employed in the proof of blow-up of the mild solution, in \cite{mueller2000critical}, can be applied to our case of interest to describe turbulence initiation under It\^o white Gaussian noise.
\smallskip

In \cite{barkley2016theoretical}, different types of noise are admitted to perturb the system, although only It\^o noise is addressed to simplify numerical simulations. In this work, we study the onset of turbulence for various Gaussian noise terms. For instance, Stratonovich noise is a natural alternative to It\^o noise in physical applications and fluid dynamics, particularly due to the chain rule property \cite{flandoli20212d,holm2019stochastic}. The computation of higher bounds for the probability of metastable transitions has already been achieved for cable equations under the assumption of additive It\^o noise \cite{berglund2023stochastic,bernuzzi2023bifurcations}. Under such assumptions, the solution to the problem displays a qualitatively different behaviour from the multiplicative noise case: first, its sign can change in contrast to the cable equation with multiplicative noise; furthermore, the laminar state is not an absorbing state, which is a fundamental property of the original system. Nevertheless, the Cole-Hopf transformation reveals structural parallels between the systems. In fact, the study of the strong solution of the original system under Stratonovich noise on a logarithmic scale, in the form of the KPZ equation, displays an additive noise term \cite{barna2019analytic,corwin2018exactly,hairer2013solving}. Standard techniques enable the construction of a lower bound to the probability of transition to turbulence in specific domain regions. This method is then extended to the case of red noise in order to introduce memory in the noise component. We discuss two different interpretations of red (Stratonovich) noise, which are known to find applications in climate science \cite{morr2022red,kuehn2021warning}.
\smallskip

In conclusion, our methods prove the possibility of transition to turbulent states under different types of noise assumptions in simplified fluid dynamics models that have been suggested in applications. These results are implied by the stochastic perturbations in the system despite the non-trivial nature of the noise involved. Furthermore, they advance the analytic study of SPDEs, particularly with (multiplicative) Stratonovich noise.
\smallskip

The paper is structured as follows. In Section \ref{sec:2}, we introduce our main mathematical tools. Moreover, we justify the linearization of the system as a method to study the solutions in the proximity to the laminar state and to obtain a lower bound to the probability of the transition to turbulence. We construct the fundamental solution of the cable equation and discuss its properties. In Section \ref{sec:3}, we consider perturbations enforced by It\^o noise. We study the mild solution of the linearized model on the laminar state. By applying a suitable operator to counter the effect of the drift component, we obtain an observable in the form of a martingale. Through the use of similar techniques employed as in \cite{mueller2000critical}, we obtain the bound to the transition to turbulence. In Section \ref{sec:4}, we consider the noise term in the Stratonovich interpretation. We assume first space-time white noise and then red noise in time. A lower bound to the local and global transition to turbulence is proven by studying its strong solution on a logarithmic scale through the Cole-Hopf transformation \cite{hairer2013solving}. The methods are also shown to be extendable to the case of space heterogeneity in the system \cite{bernuzzi2023bifurcations}. In Section \ref{sec:5}, we compare the stated methods and discuss the differences with already existing approaches.

\section{Preliminaries and linearization} \label{sec:2}

We introduce the spatial interval $[0,L]$ for $L>0$, to be interpreted as the width of a pipe. For $x\in[0,L]$ and $t>0$, we define a variable $q=q(x,t)$ modelling the turbulence level of a fluid along the plane Couette flow as discussed in~\cite{barkley2016theoretical,gome2024phase}. We study a family of stochastic partial differential equations (SPDEs) used as simplified transition-to-turbulence models, for which we provide a mathematical and physical interpretation below. The SPDEs are given by
\begin{align} \label{eq:syst_q}
    \left\{
    \begin{alignedat}{2}
    \text{d} q(x,t) = &\left( \partial_{xx}^2 q(x,t) - q(x,t) + (r+1) q(x,t)^2 \left( 2 - q(x,t) \right) + \sigma_{\text{R}} q(x,t) \circ F( \xi ) (x,t) \right) \text{d}t \\
    & + \sigma_{\text{I}} q(x,t) Q^{\frac{1}{2}} \text{d}W_t +  \sigma_{\text{S}} q(x,t) \circ Q^{\frac{1}{2}} \text{d}W_t, \\
    \partial_x q(0,t) = &\partial_x q(L,t) = 0 , \\ 
    q(x,0) = &q_0(x).
    \end{alignedat}
    \right.
\end{align}
We interpret $r>0$ as a value related to the Reynolds number, which defines the behaviour of the solution. The family of models \eqref{eq:syst_q} includes noises of different forms to consider different interpretations of the small stochastic forcing. Therefore, the following parameters are to be interpreted as intensities of the noise term and are justified by models describing turbulence: $\sigma_{\text{I}} \geq 0$ as the intensity of white noise interpreted in the It\^o sense \cite{gome2024phase}; $\sigma_{\text{S}} \geq 0$ as the intensity of white noise interpreted in the Stratonovich perspective \cite{flandoli20212d}, labeled with $\circ$; $\sigma_{\text{R}} \geq 0$ as a correlation intensity which can induce perturbations in the system \cite{kuehn2021warning,morr2022red} in Stratonovich sense. The It\^o and Stratonovich assumptions are mathematically equivalent, as they can be converted by including the It\^o-Stratonovich correction term \cite{gardiner1985handbook,twardowska2004relation}. As such, we always assume $\sigma_{\text{I}} \sigma_{\text{S}} =0$. Then, under the assumption of red noise, we introduce the adapted Ornstein-Uhlenbeck process $\xi=\xi(x,t)$ in $L^2([0,L])$ and assume $\sigma_{\text{R}} > 0$. The operator $F$ is interpreted as a differential operator that maps the Ornstein-Uhlenbeck process to $L^2([0,L])$. Examples are provided in Subsection \ref{subsec:red}. Conversely, the noise is interpreted in Section \ref{sec:3} as It\^o white noise in time, i.e., $\sigma_{\text{I}} > \sigma_{\text{S}} = \sigma_{\text{R}} = 0$, and in Subsection \ref{subsec:white} as Stratonovich white noise in time, i.e., $\sigma_{\text{S}} > \sigma_{\text{I}}=\sigma_{\text{R}} = 0$. The systems discussed in the paper are converted to the It\^o noise form in Appendix \ref{app:A}.
We introduce the cylindrical Wiener process $W$, as defined in \cite[Chapter 4]{da2014stochastic}, i.e.,
\begin{align} \label{eq:eta}
    W_t:=W(x,t)=\sum_{i=0}^\infty b_i(x) \beta_i(t) ,
\end{align}
for $\left\{\beta_i\right\}_{i\in\mathbb{N}}$ a family of independent normalized scalar Wiener processes in time, with filtration $\mathcal{F}_t^{W}$, and $\left\{b_i\right\}_{i\in\mathbb{N}}$ a basis of $L^2([0,L])$. The noise $Q^{\frac{1}{2}} W$, considered as minor fluctuations in the fluid, is then assumed to be a $Q$-Wiener process. The operator $Q$ is self-adjoint and non-negative. We define its eigenvalues as $\left\{\zeta_i\right\}_{i\in\mathbb{N}}$ with corresponding eigenbasis $\left\{b_i\right\}_{i\in\mathbb{N}}$ of $L^2([0,L])$, described further below in each section. The operator also assumes at least one of the following properties: it is trace-class, or it is bounded with eigenfunctions that converge to the eigenfunctions of the Laplacian in $L^\infty([0,L])$-norm with rate sufficient to imply the existence of the solution \cite{bernuzzi2023bifurcations}. 
The boundary conditions are specified as homogeneous Neumann, although our methods also extend to periodic boundary conditions.
Lastly, the initial condition $q_0$ is assumed to be non-negative on $[0,L]$ and a positive function almost everywhere on the interval. Therefore, we consider $q$ to be non-negative for any $t>0$ and to not be initiated in the laminar state. Under these assumptions, the solution of \eqref{eq:syst_q} is proven \cite[Chapter 7]{da2014stochastic} to be in the Hilbert space $L^2([0,L])$ almost surely for $t>0$. 
The scalar product of $L^2([0,L])$ is defined as $\langle\cdot,\cdot\rangle$ and the corresponding norm as $\lvert\lvert\cdot\rvert\rvert$. The norms in the other $L^p([0,L])$ spaces are indicated as $\lvert\lvert\cdot\rvert\rvert_p$. The scalar product in any other Hilbert space $\mathcal{X}$ is labeled as $\langle \cdot, \cdot \rangle_\mathcal{X}$.
\smallskip

The deterministic system, i.e., $\sigma_{\text{I}} = \sigma_{\text{S}} = \sigma_{\text{R}} = 0$, displays three steady states: $q_1\equiv0$, $q_2\equiv q_-$ and $q_3\equiv q_+$, for
\begin{align*}
    q_{\pm}= 1\pm \sqrt{\frac{r}{r+1}}\in[0,2] .
\end{align*}
The deterministically stable states $q_1$ and $q_3$ are identified as the laminar and the turbulent state, respectively. The state $q_2$ is a saddle for the deterministic model. Lastly, we define $\tau_J$ as the first time $t\geq0$, such that $\lvert\lvert q(\cdot,t) \rvert\rvert_\infty>J$. In Figure \ref{fig:Fig1}, we observe two examples of turbulence onset under It\^o white noise and Stratonovich white noise assumptions, respectively. In each case, the initial state lies below the saddle, $0<q_0<q_2$, and $\tau_{q_3}<T=10$. The initiation of turbulence is attributed to the growth of an $L^p([0,L])$-norm. The paths are obtained through the TAMS algorithm.

\begin{figure}[h!]
    \centering
    \subfloat[Trajectory $q$ solving \eqref{eq:syst_q}\\ with It\^o white noise\\ and describing turbulence onset\\ with respect to the $L^1{([0,L])}$-norm.]{\begin{overpic}[scale=0.30]{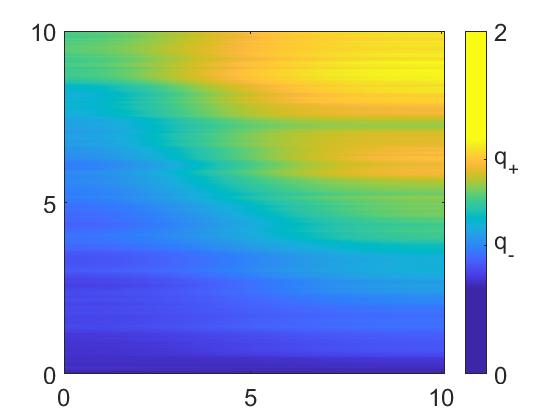}
    \put(-20,380){$t$}
    \put(530,20){$x$}
    \end{overpic}}
    \hspace{8mm}
    \subfloat[Trajectory $q$ solving \eqref{eq:syst_q}\\ with It\^o white noise\\ and describing turbulence onset\\ with respect to the $L^\infty{([0,L])}$-norm.]{\begin{overpic}[scale=0.30]{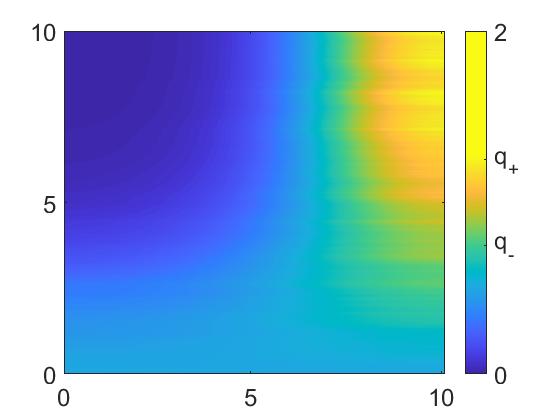}
    \put(-20,380){$t$}
    \put(530,20){$x$}
    \end{overpic}}
    
    \vspace{0mm}
    
    \subfloat[Trajectory $q$ solving \eqref{eq:syst_q}\\ with Stratonovich white noise\\ and describing turbulence onset\\ with respect to the $L^1{([0,L])}$-norm.]{\begin{overpic}[scale=0.30]{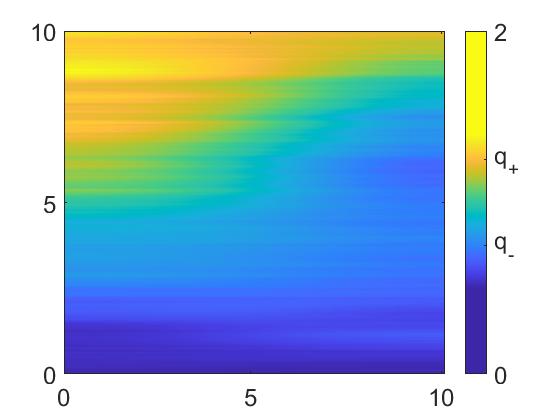}
    \put(-20,380){$t$}
    \put(530,20){$x$}
    \end{overpic}}
    \hspace{8mm}
    \subfloat[Trajectory $q$ solving \eqref{eq:syst_q}\\ with Stratonovich white noise\\ and describing turbulence onset\\ with respect to the $L^\infty{([0,L])}$-norm.]{\begin{overpic}[scale=0.30]{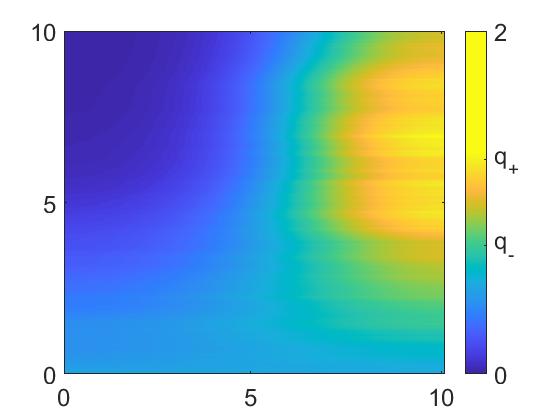}
    \put(-20,380){$t$}
    \put(530,20){$x$}
    \end{overpic}}
    
    \caption{$(a)$ and $(b)$ show trajectories of $q$, solution of \eqref{eq:syst_q}, indicating the onset of turbulence under It\^o noise, $0.5=\sigma_{\text{I}}>\sigma_{\text{S}}=\sigma_{\text{R}}=0$; whereas in $(c)$ and $(d)$ it is associated with Stratonovich noise, $0.5=\sigma_{\text{S}}>\sigma_{\text{R}}=\sigma_{\text{R}}=0$. The interpretation of turbulence initiation is associated with an $L^p([0,L])$-norm described below. \\
    We consider $\left\{e_i\right\}_{i\in\mathbb{N}}$ the normalized eigenfunctions of the Laplace operator on $[0,L]$ under Neumann boundary conditions. We assume the noise perturbation on the solution along $101$ modes, $b_i=e_i$ for $i\in\{0,\dots,100\}$, with intensity $\zeta_i=\text{exp}\left(-(i-1)^2\right)$. The initial solution is set at $q_0\equiv 0.5$. The Reynolds parameter is $r=\frac{1}{15}$, which implies that $q_-=0.75$ and $q_+=1.25$. We set $L=T=10$ and space and time step as $0.1$ and $0.01$, respectively. In $(a)$ and $(c)$ we observe the rise of $\lvert\lvert q \rvert\rvert_1$ to the value $q_+ L$, while in $(b)$ and $(d)$ we capture the rise of $\lvert\lvert q \rvert\rvert_\infty$ to the value $q_+$. These rare events are computed via the TAMS algorithm, for which we run $50$ simulations each and use the respective norm as a score function. The simulations are achieved through the discretized mild solution formula \cite{da2014stochastic}. As described in Appendix \ref{app:A}, the systems differ in view of the It\^o-Stratonovich correction term, which implies an additional heterogeneous positive drift term in the case of Stratonovich white noise.}
    \label{fig:Fig1}
\end{figure}
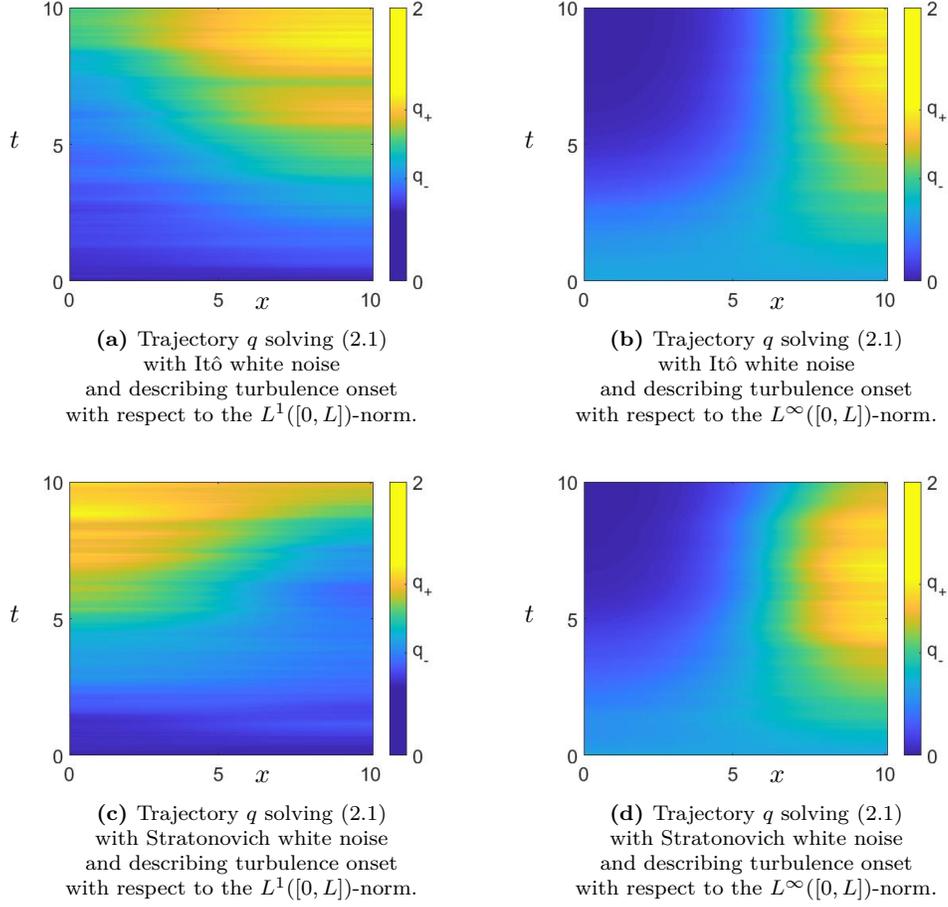 

In order to observe the behaviour of $q$ in the proximity of the laminar state, we study the mild solution $u_\alpha=u_\alpha(x,t)$ of the linear system with corresponding noise,
\begin{align} \label{eq:syst_gamma}
    \left\{
    \begin{alignedat}{2}
    \text{d} u_\alpha(x,t) = &\left( \partial_{xx}^2 u_\alpha(x,t) - \alpha u_\alpha(x,t) + \sigma_{\text{R}} u_\alpha(x,t) \circ F( \xi ) (x,t) \right) \text{d}t \\
    &+ \sigma_{\text{I}} u_\alpha(x,t) Q^{\frac{1}{2}} \text{d}W_t +  \sigma_{\text{S}} u_\alpha(x,t) \circ Q^{\frac{1}{2}} \text{d}W_t , \\
    \partial_x u_\alpha(0,t) = &\partial_x u_\alpha(L,t) = 0 , \\
    u_\alpha(x,0) = &q_0(x) ,
    \end{alignedat}
    \right.
\end{align}
for $\alpha=1$. We refer to the solution of the system following the noise assumptions, i.e., the values of $\sigma_{\text{I}},\sigma_{\text{S}},\sigma_{\text{R}}$: we denote by $u_\alpha^{\text{I}}=u_\alpha^{\text{I}}(x,t)$ the mild solution of \eqref{eq:syst_gamma} under It\^o white noise; $u_\alpha^{\text{S}}=u_\alpha^{\text{S}}(x,t)$ is the strong solution under Stratonovich white noise; $u_\alpha^{\text{R}}=u_\alpha^{\text{R}}(x,t)$ refers to the strong solution under Stratonovich red noise. We focus on the case for which $q\leq 2$, as $0\equiv q_1<q_2<q_3<2$. We define $\upsilon^{\text{S}}_J$ as the first time $t\geq0$, such that $\lvert\lvert u^{\text{S}}_\alpha(\cdot,t) \rvert\rvert_\infty>J$. Equivalently, $\upsilon^{\text{I}}_J$ is the first time $t$, such that $\lvert\lvert u^{\text{I}}_\alpha(\cdot,t) \rvert\rvert_\infty>J$. Therefore, we can prove the lemma to follow.

\begin{lm} \label{lm:q-u}
    We consider $q$, the mild solution of \eqref{eq:syst_q}, and $u_1$, the mild solution of \eqref{eq:syst_gamma} for $\alpha=1$, $x\in[0,L]$ and $t\in[0,T]$. Moreover, we assume that $q(x,t)\leq 2$ for any $x\in[0,L]$ and $t\in[0,T]$. Then, the conclusions are:
    \begin{enumerate}
        \item[(a)] the inequality $q(x,t) \geq u_1(x,t)$ holds under the same sample $W$ and for almost every $x\in[0,L]$ and $t\in[0,T]$;
        \item[(b)] for $J\leq 2$ and $\sigma_{\text{S}}>\sigma_{\text{I}}=0$, the inequality $\tau_{J} \leq \upsilon^{\text{S}}_{J}$ holds; 
        \item[(c)] for $J\leq 2$ and $\sigma_{\text{I}}>\sigma_{\text{S}}=0$, the inequality $\tau_{J} \leq \upsilon^{\text{I}}_{J}$ holds.
    \end{enumerate}
\end{lm}
\begin{proof}
    We study $\tilde{u}=q-u_1$ under the assumptions $\sigma_{\text{I}}\geq 0$, $\sigma_{\text{S}}\geq 0$ and $\sigma_{\text{R}}\geq 0$. 
    The difference $\tilde{u}$ is the mild solution of
    \begin{align*}
        \left\{
        \begin{alignedat}{2}
        \text{d} \tilde{u}(x,t) = &\left( \partial_{xx}^2 \tilde{u}(x,t)- \tilde{u}(x,t) + (r+1) q(x,t)^2 \left( 2 - q(x,t) \right) + \sigma_{\text{R}} \tilde{u}(x,t) \circ F( \xi ) (x,t) \right) \text{d}t \\
        &+ \sigma_{\text{I}} \tilde{u}(x,t) Q^{\frac{1}{2}} \text{d}W_t +  \sigma_{\text{S}} \tilde{u}(x,t) \circ Q^{\frac{1}{2}} \text{d}W_t , \\
        \partial_x \tilde{u}(0,t) = &\partial_x \tilde{u}(L,t) = 0 , \\
        \tilde{u}(x,0) \equiv &0 .
        \end{alignedat}
        \right.
    \end{align*}
    Therefore, it solves,
    \begin{align} \label{eq:mild_diff_1}
        \tilde{u}(x,t)= &(r+1) \int_0^t \txte^{(t-s) \left(\partial_{xx}^2-1\right)} \left( q(x,s)^2 \left( 2 - q(x,s) \right) \right) \text{d}s 
        + \sigma_{\text{R}} \int_0^t \txte^{(t-s) \left(\partial_{xx}^2-1\right)} \tilde{u}(x,s) \circ F( \xi ) (x,s) \text{d}s \\
        &+ \sigma_{\text{I}} \int_0^t \txte^{(t-s) \left(\partial_{xx}^2-1\right)} \tilde{u}(x,s) Q^{\frac{1}{2}} \text{d}W_s 
        + \sigma_{\text{S}} \int_0^t \txte^{(t-s) \left(\partial_{xx}^2-1\right)} \tilde{u}(x,s) \circ Q^{\frac{1}{2}} \text{d}W_s . \nonumber
    \end{align}
    Due to the Neumann boundary conditions, the continuous semigroup $\txte^{(t-s) \left(\partial_{xx}^2-1\right)}$ does not affect the sign of the argument function. Since the first term in the right-hand side of \eqref{eq:mild_diff_1} is positive and by the fact that the other terms are multiplicative in $\tilde{u}$, it follows that $q(x,t)\geq u^{\text{S}}_1(x,t)$ for any $x\in[0,L]$ and $t\in[0,T]$.
    \smallskip
    
    We set $\sigma_{\text{S}}>\sigma_{\text{I}}=0$. If $\upsilon^{\text{S}}_{J}\leq\tau_2$, then we obtain $\tau_{J}\leq\upsilon^{\text{S}}_{J}$. Conversely, for $\tau_2 \leq \upsilon^{\text{S}}_{J}$, it follows that $\tau_{J}\leq\upsilon^{\text{S}}_{J}$, since $\tau_{J}\leq\tau_2$. The It\^o noise perspective, i.e., case $(c)$, can be proven through equivalent reasoning. 
\end{proof}
The linearization enables the construction of the following results and further justifies the study of the fundamental solution of the cable equation. We shall denote the fundamendal solution by $G_\alpha$ and it solves the system

\begin{align} \label{eq:syst_G}
    \left\{\begin{alignedat}{2}
        \text{d} G_\alpha(x,y,t,[0,L]) &=
        \left( \partial_{xx}^2 - \alpha \right) G_\alpha(x,y,t,[0,L]) \text{d}t ,\\
        \partial_x G_\alpha(0,y,t,[0,L]) &= \partial_x G_\alpha(L,y,t,[0,L]) = 0 ,\\
        G_\alpha(x,y,0,[0,L])&=\delta_0(y-x) ,
    \end{alignedat}\right.
\end{align}
where $\delta_0$ is the Dirac delta. The solution $G_\alpha$ can be obtained through different methods \cite{walsh1986introduction}. For the purposes of this paper, we consider solely the form

\begin{align*}
    G_\alpha(x,y,t,[0,L]) 
    &= \txte^{-t \alpha}
    \left( \sum_{n=0}^\infty e_n(x) e_n(y) \txte^{-t \lambda_n} \right) , 
\end{align*}
for $\left\{e_i\right\}_{i\in\mathbb{N}}$ the eigenbasis of the Laplacian operator in $L^2([0,L])$ and for $\left\{-\lambda_i\right\}_{i\in\mathbb{N}}$ the corresponding eigenfunctions. Under Neumann boundary conditions, they are defined as

\begin{align*}
    \begin{alignedat}{2}
        e_0 &\equiv \frac{1}{\sqrt{L}} , &&  \\
        e_n(x) &= \sqrt{\frac{2}{L}} \text{cos}\left(\frac{n \pi x}{L}\right), && \; \text{for any $n\in\mathbb{N}_{>0}$} , \\
        \lambda_n &= \left(\frac{n \pi}{L}\right)^2 , &&\;\text{for any $n\in\mathbb{N}$} ,
    \end{alignedat}
\end{align*}
for $x\in[0,L]$. Finally, it is easy to observe that the fundamental solution satisfies
\begin{align}
    \label{eq:tools_G}
    \int_0^L G_\alpha(x,y,t,[0,L]) G_\alpha(y,z,s,[0,L]) \text{d} y = G_\alpha(x,z,t+s,[0,L]) 
    ,
\end{align}
for $t>0$, $s>0$ and $z\in[0,L]$. 

\section{It\^o noise: Countering the drift component} \label{sec:3}

In this section, we study white noise in It\^o sense, interpreted as the assumption $\sigma_{\text{I}}>\sigma_{\text{S}}=\sigma_{\text{R}}=0$ in \eqref{eq:syst_q} and in \eqref{eq:syst_gamma}. We assume $b_0=e_0$ and that $\zeta_0>0$. We consider $u^{\text{I}}_\alpha=u^{\text{I}}_\alpha(x,t)$, the mild solution of \eqref{eq:syst_gamma} for $\alpha>0$ as
\begin{align}       
    \label{eq:duhamel_form}
    u^{\text{I}}_{\alpha}(x,t) = \int_0^L G_\alpha(x,y,t,[0,L]) q_0(y) \text{d} y 
    + \sigma_{\text{I}} \int_0^L \int_0^t G_\alpha(x,y,t-s,[0,L]) u^{\text{I}}_{\alpha}(y,s) Q^{\frac{1}{2}} \text{d} W_s \text{d} y. 
\end{align}
For the fixed time $T>0$ and any $y\in[0,L]$, we define 
\begin{align*}
    \phi_\alpha(x,y,t,[0,L]):=G_\alpha(x,y,T-t,[0,L]),
\end{align*}
where $x\in[0,L]$ and $t\in[0,T]$. By construction of \eqref{eq:syst_G}, we have that $\phi_\alpha$ solves 
\begin{align*}
    \left\{\begin{alignedat}{2}
        \text{d} \phi_\alpha(x,y,t,[0,L]) &=
        \left(-\partial^2_{xx} + \alpha \right) \phi_\alpha(x,y,t,[0,L]) \text{d}t ,\\
        \partial_x \phi_\alpha(0,y,t,[0,L]) &= \partial_x \phi_\alpha(L,y,t,[0,L]) = 0 ,\\
        \phi_\alpha(x,y,T,[0,L])&=\delta_0(y-x) ,
    \end{alignedat}\right.
\end{align*}
for $x\in[0,L]$ and $t\in[0,T)$. In the following lemma, we use $\phi_\alpha$ to define an observable, which we prove to be a martingale. 

\begin{lm} \label{lm:inequality_quadratic_variation}
    For any $x\in[0,L]$ and $t \in [0,T]$, the process $M$ defined as
    \begin{align} \label{eq:construction_M}
    M(t)=\int_0^L \int_0^L \phi_\alpha(x,y,t,[0,L]) u^{\text{I}}_{\alpha} (y,t)~ \text{d} y \text{d} x
    = \txte^{-(T-t) \alpha} \lvert\lvert u^{\text{I}}_{\alpha} (\cdot,t) \rvert\rvert_1,
    \end{align}
    with $u^{\text{I}}_{\alpha}$ denoting the mild solution of \eqref{eq:syst_gamma}, is a non-negative $\mathcal{F}_t^{W}$-martingale. Its quadratic variation, $\left\langle M \right\rangle$, satisfies
    \begin{align} \label{eq:inequality_quadratic_variation}
        \left\langle M \right\rangle (t) \geq \zeta_0 \frac{\sigma_{\text{I}}^2}{L^2} \int_0^t M(s)^{2} ~\text{d} s,
    \end{align}
    for any $t \in [0,T]$.
\end{lm}
\begin{proof}
    In the first part of the proof, we follow a similar approach to \cite[Lemma 2.3]{mueller1993blowup} to prove that $M$ is a $\mathcal{F}_t^{W}$-martingale. From its construction in \eqref{eq:construction_M}, we consider the mild solution form of $u^{\text{I}}_{\alpha}$ in \eqref{eq:duhamel_form}. Then, we employ property \eqref{eq:tools_G} to obtain
    \begin{align} \label{eq:martingale_M_0}
        M(t)= \int_0^L \int_0^L \phi_\alpha(x,y,0,[0,L]) q_0(y) \text{d} y \text{d} x
        + \sigma_{\text{I}} \int_0^L \int_0^L \int_0^t \phi_\alpha(x,y,s,[0,L]) u^{\text{I}}_{\alpha}(y,s) Q^{\frac{1}{2}} \text{d} W_s \text{d} y \text{d} x .
    \end{align}
    The observable $M$ is an $\mathcal{F}_t^{W}$-martingale since it is the sum of a constant and an $\mathcal{F}_t^{W}$-martingale, which is a stochastic integral with integrand independent of $t$. Its quadratic variation is
    \begin{align} \label{eq:martingale_M}
        \left\langle M \right\rangle (t)= \sigma_{\text{I}}^2 \int_0^L \int_0^L \int_0^t \left( Q^{\frac{1}{2}} \left( \phi_\alpha(x,y,s,[0,L]) u^{\text{I}}_{\alpha}(y,s) \right) \right)^2 \text{d} s \text{d} y \text{d} x .
    \end{align}
    Inequality \eqref{eq:inequality_quadratic_variation} follows from similar steps to those included in the proof of \cite[Lemma 2]{mueller2000critical}, which we adapt here to our setting. Through Jensen's inequality, we obtain
    \begin{align*}
        &\sigma_{\text{I}}^2 \int_0^L \int_0^L \left( Q^{\frac{1}{2}} \left( \phi_\alpha(x,y,s,[0,L]) u^{\text{I}}_{\alpha}(y,s) \right) \right)^2 \text{d}y \text{d}x \\
        &= \sigma_{\text{I}}^2 L^2 \int_0^L \int_0^L \left( Q^{\frac{1}{2}} \left( \phi_\alpha(x,y,s,[0,L]) u^{\text{I}}_{\alpha}(y,s) \right) \right)^2 \frac{1}{L^2} \text{d}y \text{d}x \\
        &\geq \sigma_{\text{I}}^2 L^2 \left( \int_0^L \int_0^L Q^{\frac{1}{2}} \left( \phi_\alpha(x,y,s,[0,L]) u^{\text{I}}_{\alpha}(y,s) \right) \frac{1}{L^2} \text{d}y \text{d}x \right)^{2} \\
        &= \sigma_{\text{I}}^2 L^2 \left( \zeta_0^\frac{1}{2} \int_0^L \int_0^L \left( \phi_\alpha(x,y,s,[0,L]) u^{\text{I}}_{\alpha}(y,s) \right) \frac{1}{L^2} \text{d}y \text{d}x \right)^{2} = \zeta_0 \frac{\sigma_{\text{I}}^2}{L^2} M(s)^2 .
    \end{align*}
    The proof is concluded by integrating of the left-hand side and right-hand side over the time interval $s\in[0,t]$, for $t\in[0,T]$.
\end{proof}

As indicated in \eqref{eq:construction_M}, the observable $M$ is equivalent to the $L^1([0,L])$-norm of $u^{\text{I}}_\alpha$ rescaled in time through an exponential weight. The weight increases with time $t$, and its effect balances the dissipative term in the cable equation in \eqref{eq:syst_gamma} for $\alpha>0$. Consequently, the integrands in \eqref{eq:martingale_M_0} do not depend on $t$, and $M$ is a martingale. The simplicity of the formula is a consequence of the fact that $\lambda_0=0$ and that the corresponding eigenfunction $e_0$ is constant.
\smallskip

We set $J_1=\frac{\lvert\lvert q_0 \rvert\rvert_1}{L}$ and $J_0 \txte^{-T \alpha} < J_1 \txte^{-T \alpha} < J_1 < J_2$. From construction, we obtain $M(0)=L J_1 \txte^{-T \alpha}$. The values $J_0,J_1$ and $J_2$ are chosen to guarantee the correctness of the studied inequalities upon rescaling of $L,\sigma_{\text{I}},\alpha$ and $T$. Lastly, we define the stopping time $\tau$ as the first time in which $M(t)$ assumes the values $L J_0 \txte^{-T \alpha}$ or $L J_2$.

\begin{lm} \label{lm:galton-watson}
    For any $\alpha>0$, $\sigma_{\text{I}}>0$, $T>0$, we set $J_0,J_1$ and $J_2$ as defined above. Then, the following result holds:
    \begin{align} \label{eq:gw_unrestrained}
        \mathbb{P}\left( M(\tau) = L J_2 \right) = \frac{J_1-J_0}{J_2 \txte^{T \alpha} -J_0} .
    \end{align}
    Furthermore, we obtain the bound
    \begin{align} \label{eq:gw_restrained}
        \mathbb{P}\left( M(\tau \wedge T ) = L J_2 \right) 
        \geq \frac{J_1-J_0}{J_2 \txte^{T \alpha} -J_0} - \left(\frac{J_2}{J_0} \txte^{T \alpha} - \frac{J_1}{J_0} \right) \frac{L}{\zeta_0^\frac{1}{2} \sigma_{\text{I}}} T^{-\frac{1}{2}} ,
    \end{align}
    for $\tau \wedge T = \text{min} \left\{ \tau, T \right\}$.
\end{lm}
\begin{proof}
    Since $M$ is a martingale, we can apply Doob's optional stopping theorem to obtain that $\mathbb{E}(M(\tau))=\mathbb{E}(M(0))= L J_1 \txte^{-T \alpha}$. The definition of the stopping time $\tau$ implies that
    \begin{align*}
        \mathbb{E}(M(\tau)) = L J_2 \mathbb{P}\left( M(\tau) = L J_2 \right) + L J_0 \txte^{-T \alpha} \left( 1- \mathbb{P}\left( M(\tau) = L J_2 \right) \right) ,
    \end{align*}
    from which equality \eqref{eq:gw_unrestrained} follows. The Dubins-Schwarz theorem and the fact that $M$ is a continuous martingale entail that
    \begin{align*}
        M(t) = L J_1 \txte^{-T \alpha} + \tilde{W} \left( \left\langle M \right\rangle (t) \right) ,
    \end{align*}
    for some scalar Wiener process $\tilde{W}(t)$ and any $t\in[0,T]$. Setting the martingale in such a form implies
    \begin{align} \label{eq:long_split_1}
        \mathbb{P} \left( T<\tau \right) 
            & = \mathbb{P} \left( T<t, L J_0 \txte^{-T \alpha} < M(t) < L J_2, \text{for } t\in[0,T] \right) \nonumber\\
            & \leq \mathbb{P} \left( T<t, M(t) < L J_2, \text{for } t\in[0,T] \right)\\
            & = \mathbb{P} \left( T<t, L J_1 \txte^{-T \alpha} + \tilde{W} \left( \left\langle M \right\rangle (t) \right) < L J_2, \text{ for } t\in[0,T] \right) . \nonumber
    \end{align}
    In the case $t<\tau$, it follows from \eqref{eq:inequality_quadratic_variation} that
    \begin{align*}
        \left\langle M \right\rangle (t) \geq \zeta_0 \sigma_{\text{I}}^2 J_0^2 \txte^{-2 T \alpha} t
    \end{align*}
    and, from \eqref{eq:long_split_1}, that
    \begin{align*}
        \mathbb{P} \left( T<\tau \right) 
        & \leq \mathbb{P} \left( T<t, L J_1 \txte^{- T \alpha} + \tilde{W} \left( \left\langle M \right\rangle (t) \right) < L J_2, \text{ for } t\in[0,T] \right) \nonumber\\
        & \leq \mathbb{P} \left( T<t, L J_1 \txte^{- T \alpha} + \tilde{W} \left( t \right) < L J_2, \text{ for } t\in\left[0, \left\langle M \right\rangle (T)\right] \right)\\
        & \leq \mathbb{P} \left( T<t, L J_1 \txte^{- T \alpha} + \tilde{W} \left( t \right) < L J_2, \text{ for } t\in\left[0, \zeta_0 \sigma_{\text{I}}^2 J_0^2 \txte^{-2 T \alpha} T\right] \right) . \nonumber 
    \end{align*}
    Subsequently, we get
    \begin{align*}
        \mathbb{P} \left( T<\tau \right) 
        & \leq \mathbb{P} \left( T<t, L J_1 \txte^{- T \alpha} + \tilde{W} \left( t \right) < L J_2, \text{ for } t\in\left[0, \zeta_0 \sigma_{\text{I}}^2 J_0^2 \txte^{-2 T \alpha} T\right] \right)\\
        & \leq \mathbb{P} \left( \underset{t\in\left[0, \zeta_0 \sigma_{\text{I}}^2 J_0^2 \txte^{-2 T \alpha} T\right]}{\text{sup}} \tilde{W} \left( t \right) < L J_2 - L J_1 \txte^{- T \alpha} \right)\\
        & = 1 - \mathbb{P} \left( \underset{t\in\left[0, \zeta_0 \sigma_{\text{I}}^2 J_0^2 \txte^{-2 T \alpha} T\right]}{\text{sup}} \tilde{W} \left( t \right) \geq L J_2 - L J_1 \txte^{- T \alpha} \right) .
    \end{align*}
    Then, through the reflection principle, we obtain
    \begin{align} \label{eq:tail_of_splitted_eq}
        \mathbb{P} \left( T<\tau \right) 
        & \leq 1 - \mathbb{P} \left( \underset{t\in\left[0, \zeta_0 \sigma_{\text{I}}^2 J_0^2 \txte^{-2 T \alpha} T\right]}{\text{sup}} \tilde{W} \left( t \right) \geq L J_2 - L J_1 \txte^{- T \alpha} \right) \nonumber\\
        & = 1 - 2 \mathbb{P} \left( \tilde{W} \left( \zeta_0 \sigma_{\text{I}}^2 J_0^2 \txte^{-2 T \alpha} T \right) \geq L J_2 - L J_1 \txte^{- T \alpha} \right) \nonumber\\
        & = \mathbb{P} \left( \left\lvert \tilde{W} \left( \zeta_0 \sigma_{\text{I}}^2 J_0^2 \txte^{-2 T \alpha} T \right) \right\rvert \leq L J_2 - L J_1 \txte^{- T \alpha} \right)\\
        & = \left( 2\pi \zeta_0 \sigma_{\text{I}}^2 J_0^2 \txte^{-2 T \alpha} T \right)^{-\frac{1}{2}} \int_{- L J_2 + L J_1 \txte^{-T \alpha}}^{ L J_2 - L J_1 \txte^{-T \alpha}} \text{exp} \left( - \frac{x^2}{2 \zeta_0 \sigma_{\text{I}}^2 J_0^2 \txte^{-2 T \alpha} T} \right) \text{d} x \nonumber\\
        & \leq \left(\frac{J_2}{J_0} \txte^{ T \alpha} - \frac{J_1}{J_0} \right) \frac{L}{\zeta_0^\frac{1}{2} \sigma_{\text{I}}} T^{-\frac{1}{2}} . \nonumber
    \end{align}
    From \eqref{eq:gw_unrestrained} and \eqref{eq:tail_of_splitted_eq}, it follows that
    \begin{align*}
        \mathbb{P}\left( M\left(\tau \wedge T\right) = L J_2 \right) 
        & = \mathbb{P}\left( M\left(\tau \wedge T\right) = L J_2, T\geq \tau \right) \\
        & = \mathbb{P}\left( M\left(\tau \right) = L J_2 \right) - \mathbb{P}\left( M\left(\tau \right) = L J_2, T < \tau \right) \\
        & \geq \mathbb{P}\left( M\left(\tau \right) = L J_2 \right) - \mathbb{P}\left( T < \tau \right) \\
        & \geq \frac{J_1-J_0}{J_2 \txte^{T \alpha} -J_0} - \left(\frac{J_2}{J_0} \txte^{T \alpha} - \frac{J_1}{J_0} \right) \frac{L}{\zeta_0^\frac{1}{2} \sigma_{\text{I}}} T^{-\frac{1}{2}} ,
    \end{align*}
    which concludes the proof.
\end{proof}

The inequality \eqref{eq:gw_restrained} provides a tool to establish a lower bound to the probability of rise of the $L^1(0,L)$-norm of $u^{\text{I}}_{\alpha}$. In fact, under the assumption that $L J_2 = M\left( \tau \wedge T \right)$, we can obtain the following:

\begin{align*}
    L J_2 = M\left( \tau \wedge T \right) 
    = \text{exp}\left( -\alpha \left( T - \tau \wedge T \right) \right) \left\lvert\left\lvert u^{\text{I}}_{\alpha} (\cdot,\tau \wedge T) \right\rvert\right\rvert_1 
    \leq \underset{0 \leq t \leq T}{\text{sup}}\left\lvert\left\lvert u^{\text{I}}_{\alpha} (\cdot, t) \right\rvert\right\rvert_1 .
\end{align*}
This entails that
\begin{align} \label{eq:tool_norm_1}
    \frac{J_1-J_0}{J_2 \txte^{T \alpha} -J_0} - \left(\frac{J_2}{J_0} \txte^{T \alpha} - \frac{J_1}{J_0} \right) \frac{L}{\zeta_0^\frac{1}{2} \sigma_{\text{I}}} T^{-\frac{1}{2}} 
    \leq \mathbb{P} \left( L J_2 \leq \underset{0 \leq t \leq T}{\text{sup}}\left\lvert\left\lvert u^{\text{I}}_\alpha (\cdot, t) \right\rvert\right\rvert_1 \right) ,
\end{align}
which yields a further step towards the construction of a lower bound to the probability of the onset of turbulence. In fact, such a bound is carried over to the mild solution of \eqref{eq:syst_q} in the next corollary. Since the left-hand side can be negative for large values of $T$, the assumptions $\alpha \ll 1$ or $L \ll \sqrt{\zeta_0}\sigma_{\text{I}}$ have to be enforced to ensure positivity for large intervals in time.

\begin{cor} \label{cor:mueller}
    \begin{enumerate}
        \item[(a)] We assume that $q$, the mild solution of \eqref{eq:syst_q} with It\^o noise, satisfies $0<q(x,t)\leq 2$ for any $x\in[0,L]$ and $t\in[0,T]$. Furthermore, we set the initial conditions such that $0< J_0 \txte^{-T \alpha} < \frac{\lvert\lvert q_0 \rvert\rvert_1}{L} \txte^{-T \alpha} < \frac{\lvert\lvert q_0 \rvert\rvert_1}{L} < J_2$. Then we get
        \begin{align*}
            \frac{J_1-J_0}{J_2 \txte^{T} -J_0} - \left(\frac{J_2}{J_0} \txte^{T} - \frac{J_1}{J_0} \right) \frac{L}{\zeta_0^\frac{1}{2} \sigma_{\text{I}}} T^{-\frac{1}{2}} \leq \mathbb{P} \left( L J_2 \leq \underset{0 \leq t \leq T}{\text{sup}}\left\lvert\left\lvert q (\cdot, t) \right\rvert\right\rvert_1 \right) .
        \end{align*}
        \item[(b)] We consider $q$, the mild solution of \eqref{eq:syst_q} with It\^o noise that satisfies $0< J_0 \txte^{-T \alpha} < \frac{\lvert\lvert q_0 \rvert\rvert_1}{L} \txte^{-T \alpha} < \frac{\lvert\lvert q_0 \rvert\rvert_1}{L} < J_2 < 2$. The following inequality holds:
        \begin{align} \label{eq:in_nom_infty_1}
        \underset{0 \leq t \leq T}{\text{sup}} \left( \frac{J_1-J_0}{J_2 \txte^{t} -J_0} - \left(\frac{J_2}{J_0} \txte^{t} - \frac{J_1}{J_0} \right) \frac{L}{\zeta_0^\frac{1}{2} \sigma_{\text{I}}} t^{-\frac{1}{2}} 
        \right)
        \leq \mathbb{P} \left( J_2 \leq \underset{0 \leq t \leq T}{\text{sup}}\left\lvert\left\lvert q (\cdot, t) \right\rvert\right\rvert_\infty \right) .
        \end{align}
    \end{enumerate}
\end{cor}
\begin{proof}
    The first statement follows directly from \eqref{eq:tool_norm_1} for $\alpha=1$ and Lemma \ref{lm:q-u} $(a)$. In fact, since $q(x,t)\leq 2$ for all $x\in[0,L]$ and $t\in[0,T]$, we obtain
    \begin{align*}
        \mathbb{P} \left( L J_2 \leq \underset{0 \leq t \leq T}{\text{sup}}\left\lvert\left\lvert u^{\text{I}}_1 (\cdot, t) \right\rvert\right\rvert_1 \right)
        \leq \mathbb{P} \left( L J_2 \leq \underset{0 \leq t \leq T}{\text{sup}}\left\lvert\left\lvert q (\cdot, t) \right\rvert\right\rvert_1 \right) .
    \end{align*}
    Similarly to the latter case, \eqref{eq:tool_norm_1} for $\alpha=1$, H\"older's inequality and Lemma \ref{lm:q-u} $(c)$ imply that
    \begin{align*}
        \frac{J_1-J_0}{J_2 \txte^{T} -J_0} - \left(\frac{J_2}{J_0} \txte^{T} - \frac{J_1}{J_0} \right) \frac{L}{\zeta_0^\frac{1}{2} \sigma_{\text{I}}} T^{-\frac{1}{2}} 
        &\leq \mathbb{P} \left( L J_2 \leq \underset{0 \leq t \leq T}{\text{sup}}\left\lvert\left\lvert u^{\text{I}}_1 (\cdot, t) \right\rvert\right\rvert_1 \right) \\
        &\leq \mathbb{P} \left( J_2 \leq \underset{0 \leq t \leq T}{\text{sup}}\left\lvert\left\lvert u^{\text{I}}_1 (\cdot, t) \right\rvert\right\rvert_\infty \right)
        = \mathbb{P} \left( \upsilon^{\text{I}}_{J_2} \leq T \right) \\
        &\leq \mathbb{P} \left( \tau_{J_2} \leq T \right) 
        = \mathbb{P} \left( J_2 \leq \underset{0 \leq t \leq T}{\text{sup}}\left\lvert\left\lvert q (\cdot, t) \right\rvert\right\rvert_\infty \right)
        .
    \end{align*}
    Lastly, we obtain that, for any $T_0\leq T$,
    \begin{align*}
        \mathbb{P} \left( J_2 \leq \underset{0 \leq t \leq T_0}{\text{sup}}\left\lvert\left\lvert q (\cdot, t) \right\rvert\right\rvert_\infty \right)
        \leq \mathbb{P} \left( J_2 \leq \underset{0 \leq t \leq T}{\text{sup}}\left\lvert\left\lvert q (\cdot, t) \right\rvert\right\rvert_\infty \right)
        ,
    \end{align*}
    which concludes the proof.
\end{proof}

The statement in Corollary \ref{cor:mueller} pertains to the case when we consider trajectories of $q$ with initial conditions below the saddle state, $q_0<q_2$, and final conditions between the saddle and the turbulence state, $q_-<J_2<q_+$, thus indicating turbulence initiation. Moreover, the assumption $q(x,t)\leq 2$ for all $x\in[0,L]$ and $t\in[0,T]$ is also not required to obtain inequality \eqref{eq:in_nom_infty_1}. 
\smallskip

In this setting, the lower bound is valid (positive), for large time intervals under the assumption of thin pipes or large noise intensity. Aside from the value $J_2$ that indicates the transition to turbulence occurrence, the lower bound in \eqref{eq:gw_restrained} is affected by the choice of $J_0$. This has further implications on the role of the final time $T$ since the definition of the stopping time $\tau$ depends on $J_0$. Yet, in \eqref{eq:in_nom_infty_1}, the sign of the lower bound is not affected by the size of the time interval considered and is constant for sufficiently large $T$. This is reflected by the dissipation towards $q_1$ in the drift component of the mild solution $q$. Such a behaviour indicates that the possibility of initiation of turbulence is less likely after large times. This property is discussed further in Section \ref{sec:5}.

\section{Stratonovich noise: Comparison SPDEs in a logarithmic scale} \label{sec:4}

In this section, we set $\sigma_{\text{S}}+\sigma_{\text{R}}>\sigma_{\text{I}}=0$ as we study \eqref{eq:syst_q}. In order to enforce the existence and uniqueness of the strong solution of the system with Stratonovich noise \cite[Section 6.5]{da2014stochastic}, we enforce the following assumptions. For fixed $m\in\mathbb{N}_{>0}$, we assume that
\begin{align*}
    \zeta_i = 0 \quad , \; \text{for $i>m$} ,
\end{align*}
and that there exists an $i\in\left\{ 0,\dots, m\right\}$ such that $\zeta_i>0$ and $\left\langle e_0, b_i \right\rangle\neq 0$. Adhering to the study of the linearized system as justified in Lemma \ref{lm:q-u}, we focus on a generalized version of the system \eqref{eq:syst_gamma}. The model

\begin{align} \label{eq:syst_u}
    \left\{
    \begin{alignedat}{2}
    \text{d} u(x,t) &= \left( \partial_{xx}^2 u(x,t) - g(x) u(x,t) + \sigma_{\text{R}} u(x,t) \circ F \left( \xi \right) (x,t) \right) \text{d}t + \sigma_{\text{S}} u(x,t) \circ Q^{\frac{1}{2}} \text{d}W_t , \\
    \partial_x u(0,t) &= \partial_x u(L,t) = 0 , \\
    u(x,0) &= q_0(x) ,
    \end{alignedat}
    \right.
\end{align}
for $x\in[0,L]$ and $t>0$, accounts for space-heterogeneity through the inclusion of the non-negative function $g\in L^2([0,L])$. Such an assumption follows from the applications on which the cable equation is found, such as climate science \cite{bernuzzi2023bifurcations}, due to the recurrent discrepancies in certain domain regions that can be found in such fields. In the next subsections, we obtain a lower bound to the probability of the turbulence onset for white and red Stratonovich noise, respectively.

\subsection{White Stratonovich noise} \label{subsec:white}

We set $\sigma_{\text{R}}=0$ in order to study white Stratonovich noise in time in the equation \eqref{eq:syst_u}. The following lemma enables the study of the system on a logarithmic scale through the inverse Cole-Hopf transformation, thus obtaining the KPZ equation \cite{hairer2013solving}.

\begin{lm} \label{lm:u-v}
    We consider $u^{\text{S}}_g=u^{\text{S}}_g(x,t)$, the strong solution of \eqref{eq:syst_u} for $x\in[0,L]$ and $t>0$. Then, $v^{\text{S}}_g =v^{\text{S}}_g(x,t):=\text{log}\left( u^{\text{S}}_g \right)$ solves 
    \begin{align} \label{eq:syst_v}
        \left\{
        \begin{alignedat}{2}
        \text{d} v^{\text{S}}_g(x,t) &= \left( \partial_{xx}^2 v^{\text{S}}_g(x,t) + \left(\partial_x v^{\text{S}}_g(x,t) \right)^2 - g(x) \right) \text{d}t + \sigma_{\text{S}} Q^{\frac{1}{2}} \text{d}W_t , \\
        \partial_x v^{\text{S}}_g(0,t) &= \partial_x v^{\text{S}}_g(L,t) = 0 , \\
        v^{\text{S}}_g(x,0) &= \text{log}\left(q_0(x)\right) ,
        \end{alignedat}
        \right.
    \end{align}
    for $x\in[0,L]$ and $t>0$.
\end{lm}
\begin{proof}
    Imposing $u^{\text{S}}_g=\text{exp} \left( v^{\text{S}}_g \right)$ in \eqref{eq:syst_u}, we obtain
    \begin{align*}
        \left\{
        \begin{alignedat}{2}
        \text{d} v^{\text{S}}_g(x,t) &= u^{\text{S}}_{g}(x,t)^{-1} \circ \text{d} u^{\text{S}}_{g}(x,t) \\
        &= \left( u^{\text{S}}_{g}(x,t)^{-1} \partial_{xx}^2 u^{\text{S}}_{g}(x,t) - g(x) u^{\text{S}}_{g}(x,t) \right) \text{d}t + \left( \sigma_{\text{S}} u^{\text{S}}_{g}(x,t)^{-1} u^{\text{S}}_{g}(x,t) \right) \circ Q^{\frac{1}{2}} \text{d}W_t \\
        &= \left( \partial_{xx}^2 v^{\text{S}}_g(x,t) + \left( \partial_{x} v^{\text{S}}_g(x,t) \right)^2 - g(x) \right) \text{d}t + \sigma_{\text{S}} Q^{\frac{1}{2}} \text{d}W_t , \\
        u^{\text{S}}_{g}(0,t) \partial_x v^{\text{S}}_g(0,t) &= u^{\text{S}}_{g}(L,t) \partial_x v^{\text{S}}_g(L,t) = 0 , \\
        \text{exp}\left( v^{\text{S}}_g(x,0) \right) &= q_0(x) .
        \end{alignedat}
        \right.
    \end{align*}
    The boundary conditions in \eqref{eq:syst_v} follow from the fact that $u^{\text{S}}_g$ and $q_0$ are, by construction, almost surely positive functions for any $t>0$.
\end{proof}

The key benefit of the logarithmic perspective is the conversion of noise from a multiplicative form to an additive one. As a consequence, the nature of the drift component in \eqref{eq:syst_u} and \eqref{eq:syst_v} is drastically different. The system \eqref{eq:syst_v} is not linear since a shear deformation term is included, and a linear-in-time flow affects the solution. The lemma to follow aims to simplify the problem further.

\begin{lm} \label{lm:v-w}
    We consider $v^{\text{S}}_g$, strong solution of \eqref{eq:syst_v}, and $w^{\text{S}}_g=w^{\text{S}}_g(x,t)$, a strong solution of 
    \begin{align} \label{eq:syst_w}
        \left\{
        \begin{alignedat}{2}
        \text{d} w^{\text{S}}_g(x,t) &= \left(\partial_{xx}^2 w^{\text{S}}_g(x,t) - g(x)\right) \text{d}t + \sigma_{\text{S}} Q^{\frac{1}{2}} \text{d}W_t , \\
        \partial_x w^{\text{S}}_g(0,t) &= \partial_x w^{\text{S}}_g(L,t) = 0 , \\
        w^{\text{S}}_g(x,0) &= \text{log}\left(q_0(x)\right) ,
        \end{alignedat}
        \right.
    \end{align}
    for $x\in[0,L]$ and $t>0$. Then we obtain $v^{\text{S}}_g(x,t)\geq w^{\text{S}}_g(x,t)$ under the same sample of $W$, for any almost $x\in[0,L]$ and $t>0$.
\end{lm}
\begin{proof}
    We define $\tilde{w}=v^{\text{S}}_g-w^{\text{S}}_g$. By construction it solves
    \begin{align*}
        \left\{
        \begin{alignedat}{2}
        \partial_t \tilde{w}(x,t) &= \partial_{xx}^2 \tilde{w}(x,t) + \left( \partial_x v^{\text{S}}_g(x,t) \right)^2 , \\
        \partial_x \tilde{w}(0,t) &= \partial_x \tilde{w}(L,t) = 0 , \\
        \tilde{w}(x,0) &\equiv 0 .
        \end{alignedat}
        \right.
    \end{align*}
    The conclusion of the proof follows the same reasoning as Lemma \ref{lm:q-u}.
\end{proof}

For the existence of $w^{\text{S}}_g$, we refer to \cite{bernuzzi2023bifurcations} and \cite[Theorem 5.29]{da2014stochastic}. System \eqref{eq:syst_w} can be easily observed along the elements of a basis in $L^2([0,L])$. In the following lemma, the projections of the strong solution $w^{\text{S}}_g$ along the eigenbasis of the Laplacian operator are studied.

\begin{lm} \label{lm:w-I}
    We consider $w^{\text{S}}_g$, strong solution of \eqref{eq:syst_w} for $x\in[0,L]$ and $t>0$. Then, for all $n\in\mathbb{N}$, the scalar product $I_{n,g}(t) :=\left\langle e_n, w^{\text{S}}_g(\cdot,t) \right\rangle$ solves 
    \begin{align} \label{eq:syst_I}
        \left\{
        \begin{alignedat}{2}
        \text{d} I_{n,g}(t) &= \left( - \lambda_n I_{n,g}(t) - \left\langle e_n, g \right\rangle \right) \text{d}t + \sigma_{\text{S}} \sum_{i=0}^{m} \left( \zeta_i^{\frac{1}{2}} \left\langle e_n, b_i \right\rangle \text{d}\beta_i(t) \right) , \\
        I_{n,g}(0) &= \left\langle e_n, \text{log}\left(q_0\right) \right\rangle ,
        \end{alignedat}
        \right.
    \end{align}
    for any $t>0$.
\end{lm}
\begin{proof}
    We arbitrarily fix $n\in\mathbb{N}$. Through \eqref{eq:syst_w}, we obtain
    \begin{align*}
        \left\{
        \begin{alignedat}{2}
        \text{d} \left\langle e_n, w^{\text{S}}_g(\cdot,t) \right\rangle &= \left( \left\langle e_n, \partial_{xx}^2 w^{\text{S}}_g(\cdot,t)  \right\rangle - \left\langle e_n, g \right\rangle \right) \text{d}t + \sigma_{\text{S}}\left\langle Q^{\frac{1}{2}} e_n, \text{d}W_t \right\rangle\\
        &= \left( - \lambda_n \left\langle e_n, w^{\text{S}}_g(\cdot,t) \right\rangle -  \left\langle e_n, g \right\rangle \right) \text{d}t + \sigma_{\text{S}}\left\langle \sum_{i=0}^m \zeta_i^{\frac{1}{2}} \left\langle e_n,b_i \right\rangle b_i , \text{d}W_t \right\rangle \\
        &= \left(- \lambda_n \left\langle e_n, w^{\text{S}}_g(\cdot,t) \right\rangle -  \left\langle e_n, g \right\rangle \right) \text{d}t + \sigma_{\text{S}} \sum_{i=0}^m \left( \zeta_i^{\frac{1}{2}} \left\langle e_n, b_i \right\rangle \text{d}\beta_i(t) \right) , \\
        I_{n,g}(0) &= \left\langle e_n, \text{log}\left(q_0\right) \right\rangle  ,
        \end{alignedat}
        \right.
    \end{align*}
    for $t>0$, which concludes the proof.
\end{proof}

The previous lemmas in this section describe, in their entirety, a chain of inequalities that constitute a bound from below of the strong solution $u^{\text{S}}_g$ of system \eqref{eq:syst_u}. This approach is employed in the following theorem to define a bound to the probability of growth in a fixed time of $u^{\text{S}}_g$ on regions of the domain.

\begin{thm} \label{thm:het_u-I}
    Under the assumption that $u^{\text{S}}_g$ is the strong solution of \eqref{eq:syst_u} for $x\in[0,L]$ and $t>0$, the following inequality holds for any non-negative function $f\in L^2([0,L])$ that is not almost everywhere zero:
    \begin{align}  \label{eq:ineq_thm_u-I}
         1-\Phi\left( \frac{\lvert\lvert f \rvert\rvert_1 \text{log} \left( \lvert\lvert f \rvert\rvert_1^{-1} J' \right) - \left( \left\langle \txte^{t \partial_{xx}^2} f, \text{log} \left( q_0 \right) \right\rangle - a_0 \left\langle e_0, g \right\rangle t - \overset{\infty}{\underset{n=1}{\sum}} a_n \frac{\left\langle e_n, g \right\rangle}{\lambda_n} \left( 1 - \txte^{-t \lambda_n} \right) \right)}{ \sigma_{\text{S}} \left( \overset{m}{\underset{i=0}{\sum}} \zeta_{i} \left( a_0^2 \left\langle e_{0}, b_{i} \right\rangle^2 t + \underset{({n_1},{n_2})\neq(0,0)}{\sum} a_{n_1} a_{n_2} \left( \frac{1-\txte^{-t \left(\lambda_{n_1}+\lambda_{n_2}\right)}}{\lambda_{n_1}+\lambda_{n_2}} \right) \left\langle e_{n_1}, b_{i} \right\rangle \left\langle e_{n_2}, b_{i} \right\rangle \right) \right)^{\frac{1}{2}}} \right) 
         \leq \mathbb{P} \left( J'\leq \left\langle f, u^{\text{S}}_g(\cdot,t) \right\rangle \right) ,
    \end{align}
    for $\Phi$, the cumulative distribution function of a standard normal distributed random variable, $J'>0$, and $a_n=\left\langle e_n, f \right\rangle$ for any $n\in\mathbb{N}$.
\end{thm}
\begin{proof}
    We know that
    \begin{align*}
        I_{n,g}(t)= \txte^{-t \lambda_n} I_{n,g}(0) - \frac{\left\langle e_n, g \right\rangle}{\lambda_n} \left( 1 - \txte^{-t \lambda_n} \right) 
        + \sigma_{\text{S}} \sum_{i=0}^m \left( \zeta_i^{\frac{1}{2}} \left\langle e_n, b_i \right\rangle \int_0^t \txte^{-(t-s) \lambda_n} \text{d} \beta_i(s) \right)
    \end{align*}
    holds for any $n\in\mathbb{N}_{>0}$, and that
    \begin{align*}
        I_{0,g}(t)= I_{0,g}(0) - \left\langle e_0, g \right\rangle t 
        + \sigma_{\text{S}} \sum_{i=0}^m \left( \zeta_i^{\frac{1}{2}} \left\langle e_0, b_i \right\rangle \beta_i(s) \right) .
    \end{align*}
    It follows that
    \begin{align*}
        \mathbb{E} \left( I_{n,g}(t) \right) =
         \txte^{-t \lambda_n} I_{n,g}(0) - \frac{\left\langle e_n, g \right\rangle}{\lambda_n} \left( 1 - \txte^{-t \lambda_n} \right) ,
    \end{align*}
    for any $n\in\mathbb{N}_{>0}$, and also
    \begin{align*}
        \mathbb{E} \left( I_{0,g}(t) \right) =
         I_{0,g}(0) - \left\langle e_0, g \right\rangle t .
    \end{align*}
    Moreover, from the construction of the systems \eqref{eq:syst_I} and the limit
    \begin{align*}
        \underset{n\rightarrow\infty}{\text{lim}}\frac{\lambda_{n}}{1+n^2}<\infty ,
    \end{align*} 
    we obtain that
    \begin{align*}
        &\mathbb{E} \left( \sum_{n=0}^\infty a_n I_{n,g}(t) \right) 
        = a_0 \left( I_{0,g}(0) - \left\langle e_0, g \right\rangle t\right) + \sum_{n=1}^\infty a_n \left( \txte^{-t \lambda_n} I_{n,g}(0) - \frac{\left\langle e_n, g \right\rangle}{\lambda_n} \left( 1 - \txte^{-t \lambda_n} \right) \right) < \infty .
    \end{align*}
    Levy's continuity lemma implies that $\sum_{n=0}^\infty a_n I_{n,g}(t)$ has a Gaussian distribution with variance
    \begin{align*}
        \text{Var} \left( \sum_{n=0}^\infty a_n I_{n,g}(t) \right) 
        &= \text{Var} \left( \left\langle f, w^{\text{S}}_g (\cdot, t ) \right\rangle \right)
        =\sigma_{\text{S}}^2 \int_0^t \left\langle f, \txte^{s \partial_{xx}^2} Q \txte^{s \partial_{xx}^2} f \right\rangle \text{d}s \\
        &= \sigma_{\text{S}}^2 \int_0^t \left\langle \sum_{{n_1}=0}^\infty a_{n_1} \txte^{-s \lambda_{n_1}} e_{n_1}, Q \sum_{{n_2}=0}^\infty a_{n_2} \txte^{-s \lambda_{n_2}} e_{n_2} \right\rangle \text{d}s \\
        &= \sigma_{\text{S}}^2 \int_0^t \left\langle \sum_{{n_1}=0}^\infty \sum_{i_1=0}^m a_{n_1} \txte^{-s \lambda_{n_1}} \left\langle e_{n_1}, b_{i_1} \right\rangle b_{i_1}, \sum_{{n_2}=0}^\infty \sum_{i_2=0}^m \zeta_{i_2} a_{n_2} \txte^{-s \lambda_{n_2}} \left\langle e_{n_2}, b_{i_2} \right\rangle b_{i_2} \right\rangle \text{d}s \\
        &=  \sigma_{\text{S}}^2 \int_0^t \sum_{{n_1}=0}^\infty \sum_{{n_2}=0}^\infty  \sum_{i=0}^m a_{n_1} a_{n_2} \zeta_{i} \txte^{-s \left(\lambda_{n_1}+\lambda_{n_2}\right)} \left\langle e_{n_1}, b_{i} \right\rangle \left\langle e_{n_2}, b_{i} \right\rangle \text{d}s \\
        &=  \sigma_{\text{S}}^2 \sum_{i=0}^m \zeta_{i} \left( a_0^2 \left\langle e_{0}, b_{i} \right\rangle^2 t + \sum_{({n_1},{n_2})\neq (0,0)} a_{n_1} a_{n_2} \left( \frac{1-\txte^{-t \left(\lambda_{n_1}+\lambda_{n_2}\right) }}{\lambda_{n_1}+\lambda_{n_2}} \right) \left\langle e_{n_1}, b_{i} \right\rangle \left\langle e_{n_2}, b_{i} \right\rangle \right) ,
    \end{align*}
    and that
    \begin{align*}
        \mathbb{P}\left( \sum_{n=0}^\infty a_n I_{n,g}(t) \geq J'' \right) 
        &= 1-\Phi\left( \frac{J'' - \mathbb{E} \left( \sum_{n=0}^\infty a_n I_{n,g}(t) \right)}{ \left( \text{Var} \left( \sum_{n=0}^\infty a_n I_{n,g}(t) \right) \right)^{\frac{1}{2}}} \right) \\
        &= 1-\Phi\left( \frac{J'' - \left(a_0 \left( I_{0,g}(0) - \left\langle e_0, g \right\rangle t\right) + \sum_{n=1}^\infty a_n \left( \txte^{-t \lambda_n} I_{n,g}(0) - \frac{\left\langle e_n, g \right\rangle}{\lambda_n} \left( 1 - \txte^{-t \lambda_n} \right) \right)  \right)}{\sigma_{\text{S}} \left( \sum_{i=0}^m \zeta_{i} \left( a_0^2 \left\langle e_{0}, b_{i} \right\rangle^2 t + \sum_{({n_1},{n_2})\neq (0,0)} a_{n_1} a_{n_2} \left( \frac{1-\txte^{-t \left(\lambda_{n_1}+\lambda_{n_2}\right)}}{\lambda_{n_1}+\lambda_{n_2}} \right) \left\langle e_{n_1}, b_{i} \right\rangle \left\langle e_{n_2}, b_{i} \right\rangle \right) \right)^{\frac{1}{2}}} \right) \\
        &= 1-\Phi\left( \frac{J'' - \left( \left\langle \txte^{t \partial_{xx}^2} f, w^{\text{S}}_g(\cdot,0) \right\rangle - a_0 \left\langle e_0, g \right\rangle t - \sum_{n=1}^\infty a_n \frac{\left\langle e_n, g \right\rangle}{\lambda_n} \left( 1 - \txte^{-t \lambda_n} \right) \right)}{\sigma_{\text{S}} \left( \sum_{i=0}^m \zeta_{i} \left( a_0^2 \left\langle e_{0}, b_{i} \right\rangle^2 t + \sum_{({n_1},{n_2})\neq (0,0)} a_{n_1} a_{n_2} \left( \frac{1-\txte^{-t \left(\lambda_{n_1}+\lambda_{n_2}\right)}}{\lambda_{n_1}+\lambda_{n_2}} \right) \left\langle e_{n_1}, b_{i} \right\rangle \left\langle e_{n_2}, b_{i} \right\rangle \right) \right)^{\frac{1}{2}}} \right) ,
    \end{align*}
    for $J''\in\mathbb{R}$. We assume henceforth that $\overset{\infty}{\underset{n=0}{\sum}} a_n \; I_{n,g}(t) \geq J''$. We employ, in order, Lemma \ref{lm:w-I}, Lemma \ref{lm:v-w}, Jensen's inequality and Lemma \ref{lm:u-v} as follows:
    \begin{align*}
         \lvert\lvert f \rvert\rvert_1 \text{exp} \left( \lvert\lvert f \rvert\rvert_1^{-1} J'' \right) 
         & \leq \lvert\lvert f \rvert\rvert_1 \text{exp} \left( \lvert\lvert f \rvert\rvert_1^{-1} \sum_{n=0}^\infty a_n \; I_{n,g}(t) \right)
         = \lvert\lvert f \rvert\rvert_1 \text{exp} \left( \lvert\lvert f \rvert\rvert_1^{-1} \left\langle f, w^{\text{S}}_g(\cdot,t) \right\rangle \right) \\
         & = \lvert\lvert f \rvert\rvert_1 \text{exp} \left( \lvert\lvert f \rvert\rvert_1^{-1} \int_0^L f(x) \; w^{\text{S}}_g(x,t) \text{d}x  \right)
         \leq \lvert\lvert f \rvert\rvert_1 \text{exp} \left( \lvert\lvert f \rvert\rvert_1^{-1} \int_0^L f(x)\; v^{\text{S}}_g(x,t) \text{d}x  \right) \\
         & \leq \frac{\lvert\lvert f \rvert\rvert_1}{\lvert\lvert f \rvert\rvert_1} \int_0^L f(x)\; \text{exp} \left( v^{\text{S}}_g(x,t) \right) \text{d}x
         = \left\langle f, \text{exp} \left( v^{\text{S}}_g(\cdot,t) \right) \right\rangle
         = \left\langle f, u^{\text{S}}_g(\cdot,t) \right\rangle .
    \end{align*}
    This entails that
    \begin{align*}
         \mathbb{P}\left( \sum_{n=0}^\infty a_n I_{n,g}(t) \geq J'' \right) \leq \mathbb{P} \left( \lvert\lvert f \rvert\rvert_1 \text{exp} \left( \lvert\lvert f \rvert\rvert_1^{-1} J'' \right)
         \leq \left\langle f, u^{\text{S}}_g(\cdot,t) \right\rangle \right) .
    \end{align*}
    The proof is concluded upon defining $J'=\lvert\lvert f \rvert\rvert_1 \text{exp} \left( \lvert\lvert f \rvert\rvert_1^{-1} J'' \right)$.
\end{proof}

Theorem \ref{thm:het_u-I} provides a lower bound to the probability of the onset of $u^{\text{S}}_g$ under the assumption of heterogeneity in space, induced by the term $g$ in \eqref{eq:syst_u}. The bound highly depends on the choice of the function $f$, upon which the strong solution $u^{\text{S}}_g$ is projected. While this function defines the observable and can, therefore, be assumed to be known in applications, the shape of the function $g$ is also required to compute the bound numerically. A more in-depth discussion can be found in \cite{bernuzzi2023bifurcations}.
\smallskip

In the next corollary, we assume $g\equiv 1$ in order to extend the statement of Theorem \ref{thm:het_u-I} to the study of $q$, the strong solution of system \eqref{eq:syst_q}. The results provide a lower bound to the local initiation of turbulence in $q$.

\begin{cor} \label{cor:q-I}
    Under the assumption that $q$, strong solution of \eqref{eq:syst_q}, satisfies $0<q(x,t)\leq 2$ for any $x\in[0,L]$ and $t\in[0,T]$, the following inequality holds for any non-negative function $f\in L^2([0,L])$ that is not almost everywhere zero:
    \begin{align*}
         1-\underset{0\leq t \leq T}{\text{min}}\Phi\left( \frac{\lvert\lvert f \rvert\rvert_1 \text{log} \left( \lvert\lvert f \rvert\rvert_1^{-1} J' \right) - \left\langle \txte^{t \partial_{xx}^2} f, \text{log} \left( q_0 \right) \right\rangle + a_0 L^{\frac{1}{2}} t }{ \sigma_{\text{S}} \left( \overset{m}{\underset{i=0}{\sum}} \zeta_{i} \left( a_0^2 \left\langle e_{0}, b_{i} \right\rangle^2 t + \underset{({n_1},{n_2})\neq(0,0)}{\sum} a_{n_1} a_{n_2} \left( \frac{1-\txte^{-t \left(\lambda_{n_1}+\lambda_{n_2}\right)}}{\lambda_{n_1}+\lambda_{n_2}} \right) \left\langle e_{n_1}, b_{i} \right\rangle \left\langle e_{n_2}, b_{i} \right\rangle \right) \right)^{\frac{1}{2}}} \right) 
         \leq \mathbb{P} \left( J'\leq \underset{0\leq t \leq T}{\text{sup}} \left\langle f, q(\cdot,t) \right\rangle \right) ,
    \end{align*}
    for $\Phi$, the cumulative distribution function of a standard normal distributed random variable, $a_n=\left\langle e_n, f \right\rangle$ for any $n\in\mathbb{N}$, $g= L^{\frac{1}{2}} e_0 \equiv 1$ and $J'>0$.
\end{cor}
\begin{proof}
    For any $t\in[0,T]$, we employ Lemma \ref{lm:q-u} to obtain 
    \begin{align*}
         \mathbb{P} \left( J'\leq \left\langle f, u^{\text{S}}_g(\cdot,t) \right\rangle \right) 
         \leq \mathbb{P} \left( J'\leq \left\langle f, q(\cdot,t) \right\rangle \right) 
         \leq \mathbb{P} \left( J'\leq \underset{0\leq t \leq T}{\text{sup}} \left\langle f, q(\cdot,t) \right\rangle \right) .
    \end{align*}
    This implies that
    \begin{align*}
         \underset{0\leq t \leq T}{\text{max}} \mathbb{P} \left( J'\leq \left\langle f, u^{\text{S}}_g(\cdot,t) \right\rangle \right)
         \leq \mathbb{P} \left( J'\leq \underset{0\leq t \leq T}{\text{sup}} \left\langle f, q(\cdot,t) \right\rangle \right) .
    \end{align*}
    The statement of Theorem \ref{thm:het_u-I} concludes the proof.
\end{proof}

Similarly to Corollary \ref{cor:mueller} $(b)$, we discuss in the corollary to follow the rise of turbulence on the whole domain. Consequently, it does not require the upper bound $q\leq 2$ assumption imposed in Corollary \ref{cor:q-I}. The corollaries are further compared in Section \ref{sec:5}.

\begin{cor} \label{cor:q-I_sup}
    For $q$, strong solution of \eqref{eq:syst_q} for any $x\in[0,L]$ and $t\in[0,T]$, the following inequality holds:
    \begin{align*}
        1-\underset{0\leq t \leq T}{\text{min}} \Phi\left( \frac{L^{\frac{1}{2}} }{\sigma_{\text{S}} \left( \overset{m}{\underset{i=0}{\sum}} \zeta_{i} \left\langle e_{0}, b_{i} \right\rangle^2 \right)^{\frac{1}{2}} }
        \left( t^{-\frac{1}{2}} \left( \text{log}\left( J \right) - L^{-1} \int_0^L \text{log} \left( q_0(x) \right) \text{d}x \right) + t^{\frac{1}{2}} \right) \right) 
        \leq \mathbb{P} \left( J \leq \underset{0\leq t \leq T}{\text{sup}} \lvert\lvert q(\cdot,t) \rvert\rvert_\infty \right) ,
    \end{align*}
    for $\Phi$, the cumulative distribution function of a standard normal distributed random variable and $0<J<2$.
\end{cor}
\begin{proof}
    We consider inequality \eqref{eq:ineq_thm_u-I} for $g=f=L^{\frac{1}{2}} e_0 \equiv 1$ and $J'= L J$. This choice of $f\in L^2([0,L])$ implies that 
    \begin{align*}
        \left\langle f, u^{\text{S}}_1(\cdot,t) \right\rangle = \lvert\lvert u^{\text{S}}_1(\cdot,t) \rvert\rvert_1 ,
    \end{align*}
    for any $t>0$. Therefore, we obtain
    \begin{align*}
        1-\underset{0\leq t \leq T}{\text{min}}\Phi\left( \frac{L \; \text{log} \left( J \right) - \int_0^L \text{log} \left( q_0(x) \right) \text{d}x + L t }{ \sigma_{\text{S}} \left( L t \right)^\frac{1}{2} \left( \overset{m}{\underset{i=0}{\sum}} \zeta_{i} \left\langle e_{0}, b_{i} \right\rangle^2 \right)^{\frac{1}{2}}} \right) 
        \leq \mathbb{P} \left( L J \leq \underset{0\leq t \leq T}{\text{sup}} \lvert\lvert u^{\text{S}}_1(\cdot,t) \rvert\rvert_1 \right) .
    \end{align*}
    Lastly, we employ H\"older's inequality and Lemma \ref{lm:q-u} in
    \begin{align*}
        \mathbb{P} \left( L J \leq \underset{0\leq t \leq T}{\text{sup}} \lvert\lvert u^{\text{S}}_1(\cdot,t) \rvert\rvert_1 \right) 
        \leq \mathbb{P} \left( J \leq \underset{0\leq t \leq T}{\text{sup}} \lvert\lvert u^{\text{S}}_1(\cdot,t) \rvert\rvert_\infty \right) 
        \leq \mathbb{P} \left( J \leq \underset{0\leq t \leq T}{\text{sup}} \lvert\lvert q(\cdot,t) \rvert\rvert_\infty \right) ,
    \end{align*}
    which proves the statement.
\end{proof}

\subsection{Red Stratonovich noise} \label{subsec:red}
    
    In the previous subsection, we introduce an approach for the estimation of a lower bound to the probability of jump in \eqref{eq:syst_u} under the assumption of white Stratonovich noise, i.e., $\sigma_{\text{S}}>\sigma_{\text{R}}=0$. Conversely, we consider in the remainder part of the section the red noise assumption, i.e., $\sigma_{\text{R}}>\sigma_{\text{S}}=0$. The red noise influence can be interpreted in different manners depending on the operator $F$ and the adapted process $\xi=\xi(x,t)$ \cite{kuehn2021warning,morr2024detection,morr2022red,sardeshmukh2003drifts}. We consider the strong solution of
        \begin{align} \label{eq:syst_z}
            \left\{
            \begin{alignedat}{2}
            \text{d} \xi(x,t) &= -\kappa \xi(x,t) \text{d}t + \sigma_{\xi} Q^{\frac{1}{2}} \text{d}W_t' , \\
            \xi(x,0) &\equiv 0,
            \end{alignedat}
            \right.
        \end{align}
    for $\kappa>0$, $\sigma_{\xi}>0$, $x\in[0,L]$ and $t\geq 0$. For simplicity of notation, the noise $W'_t$ is cylindrical, adapted, and independent from $W_t$. Such an Ornstein-Uhlenbeck process enables the construction of $q$, the strong solution of \eqref{eq:syst_q}, and of $u^{\text{R}}_{g}=u^{\text{R}}_{g}(x,t)$, the strong solution of \eqref{eq:syst_u} for $x\in[0,L]$ and $t\geq 0$. We define then the inverse Cole-Hopf transform $v^{\text{R}}_g:=\text{log}\left( u^{\text{R}}_{g} \right)$. Since the noise is of Stratonovich type and we are considering strong solutions of the involved systems \cite[Theorem 5.29]{da2014stochastic}, the process solves
    \begin{align*}
        \left\{
        \begin{alignedat}{2}
        \partial_t v^{\text{R}}_g(x,t) &= \partial_{xx}^2 v^{\text{R}}_g(x,t) + \left( \partial_{x} v^{\text{R}}_g(x,t) \right)^2- g(x) + \sigma_{\text{R}} \circ F( \xi ) (x,t) , \\
        \partial_x v^{\text{R}}_g(0,t) &= \partial_x v^{\text{R}}_g(L,t) = 0 , \\
        v^{\text{R}}_g(x,0) &= \text{log}\left(q_0(x)\right) ,
        \end{alignedat}
        \right.
    \end{align*}
    for $x\in[0,L]$ and $t\geq 0$. Such a statement can be proven following the steps in the proof of Lemma \ref{lm:u-v} on the couple
    \begin{align*}
         \begin{pmatrix} v^{\text{R}}_g(x,t) \\ \xi(x,t) \end{pmatrix} = \begin{pmatrix} \text{log}\left( u^{\text{R}}_{g} \right) \\ \xi(x,t) \end{pmatrix} .
    \end{align*}
    Similarly, we note that the inequality
    \begin{align*}
        v^{\text{R}}_g(x,t) \geq w^{\text{R}}_g(x,t) 
    \end{align*}
    holds for $w^{\text{R}}_g=w^{\text{R}}_g(x,t)$, the strong solution of
    \begin{align} \label{eq:syst_w_SR}
        \left\{
        \begin{alignedat}{2}
        \partial_t w^{\text{R}}_g(x,t) &= \partial_{xx}^2 w^{\text{R}}_g(x,t) - g(x) + \sigma_{\text{R}} F( \xi ) (x,t) , \\
        \partial_x w^{\text{R}}_g(0,t) &= \partial_x w^{\text{R}}_g(L,t) = 0 , \\
        w^{\text{R}}_g(x,0) &= \text{log}\left(q_0(x)\right) ,
        \end{alignedat}
        \right.
    \end{align}
    following the reasoning of the proof in Lemma \ref{lm:v-w}. We then focus on two interpretations of red noise influence. For $F=\operatorname{Id}$, the identity operator on $L^2([0,L])$, the systems \eqref{eq:syst_z} and \eqref{eq:syst_w_SR} can be rewritten as
    \begin{align} \label{eq:syst_w_SR1}
        \left\{
        \begin{alignedat}{2}
        \text{d} \begin{pmatrix}w^{\text{R}}_g(x,t) \\ \xi(x,t) \end{pmatrix} &= 
        \left( \begin{pmatrix}
            \partial^2_{xx} & \sigma_{\text{R}} \\
            0 & -\kappa
        \end{pmatrix} \begin{pmatrix}w^{\text{R}}_g(x,t) \\ \xi(x,t) \end{pmatrix} 
        - \begin{pmatrix} g(x) \\ 0 \end{pmatrix} \right) \text{d}t
        + \begin{pmatrix} 0 & 0 \\ 0 & \sigma_{\xi} Q^{\frac{1}{2}} \end{pmatrix} \begin{pmatrix} \text{d}W_t \\ \text{d}W'_t \end{pmatrix} , \\
        w^{\text{R}}_g(x,0) &= \text{log}\left(q_0(x)\right) , \\
        \xi(x,0) &\equiv 0 ;
        \end{alignedat}
        \right.
    \end{align}
    for $F= \partial_t$, they can be interpreted as
    \begin{align} \label{eq:syst_w_SR2}
        \left\{
        \begin{alignedat}{2}
        \text{d} \begin{pmatrix}w^{\text{R}}_g(x,t) \\ \xi(x,t) \end{pmatrix} &= 
        \left( \begin{pmatrix}
            \partial^2_{xx} & -\kappa \sigma_{\text{R}} \\
            0 & -\kappa
        \end{pmatrix} \begin{pmatrix}w^{\text{R}}_g(x,t) \\ \xi(x,t) \end{pmatrix} 
        - \begin{pmatrix} g(x) \\ 0 \end{pmatrix} \right) \text{d}t
        + \begin{pmatrix} 0 & \sigma_{\text{R}} \sigma_{\xi} Q^{\frac{1}{2}} \\ 0 & \sigma_{\xi} Q^{\frac{1}{2}} \end{pmatrix} \begin{pmatrix} \text{d}W_t \\ \text{d}W'_t \end{pmatrix} , \\
        w^{\text{R}}_g(x,0) &= \text{log}\left(q_0(x)\right) , \\
        \xi(x,0) &\equiv 0 ,
        \end{alignedat}
        \right.
    \end{align}
    for $x\in[0,L]$ and $t\geq 0$. We focus first on \eqref{eq:syst_w_SR1}; nevertheless, the systems, although qualitatively different, can be studied in a similar manner. In fact, the turbulence system \eqref{eq:syst_q} corresponding to the parameters and operator $F$ associated to \eqref{eq:syst_w_SR1} displays additive red noise in $L^2([0,L])\times L^2([0,L])$, which entails that the system \eqref{eq:syst_q} has a null It\^o-Stratonovich correction term. Conversely, for parameters and operator $F$ corresponding to \eqref{eq:syst_w_SR2}, the noise in system \eqref{eq:syst_q} is interpreted in the Stratonovich sense to implement the chain rule, in contrast to It\^o's lemma \cite{brzezniak2008ito}, further below. In Appendix \ref{app:A}, the system \eqref{eq:syst_q} is converted to the It\^o perspective for the cases covered in the paper. The equations are simulated in Figure \ref{fig:Fig2}, where the initiation of turbulence is observed under different assumptions through the TAMS algorithm. The results corresponding to \eqref{eq:syst_w_SR2} are displayed at the end of the subsection, discussed in Section \ref{sec:5} and proven in Appendix \ref{app:B}. 

    \begin{figure}[h!]
        \centering
        \subfloat[Trajectory $q$ solving \eqref{eq:syst_q}\\ with $F=\operatorname{Id}$ and $\kappa=0.5$.\\ It shows turbulence onset\\ with respect to the $L^1{([0,L])}$-norm.]{\begin{overpic}[scale=0.30]{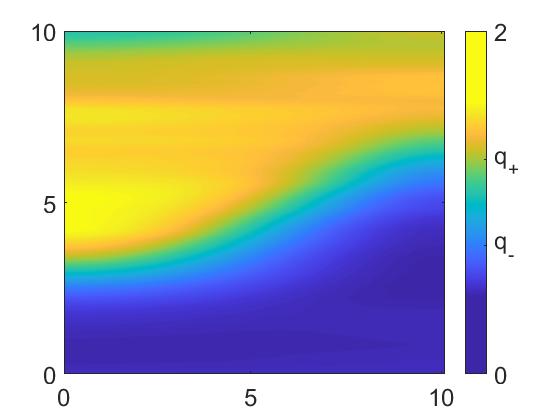}
        \put(-20,380){$t$}
        \put(530,20){$x$}
        \end{overpic}}
        \hspace{8mm}
        \subfloat[Trajectory $q$ solving \eqref{eq:syst_q}\\ with $F=\operatorname{Id}$ and $\kappa=0.5$.\\ It shows turbulence onset\\ with respect to the $L^\infty{([0,L])}$-norm.]{\begin{overpic}[scale=0.30]{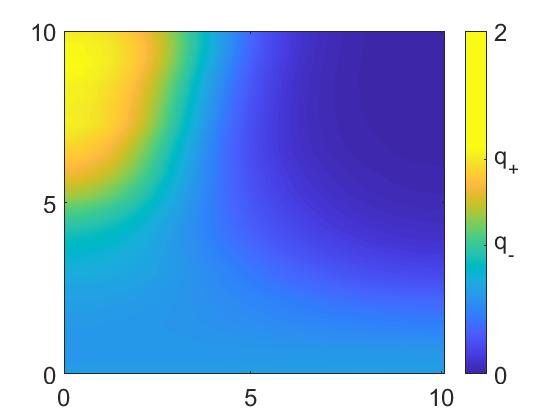}
        \put(-20,380){$t$}
        \put(530,20){$x$}
        \end{overpic}}
        
        \vspace{0mm}
        
        \subfloat[Trajectory $q$ solving \eqref{eq:syst_q}\\ with $F= \partial_t$ and $\kappa=0.05$.\\ It shows turbulence onset\\ with respect to the $L^1{([0,L])}$-norm.]{\begin{overpic}[scale=0.30]{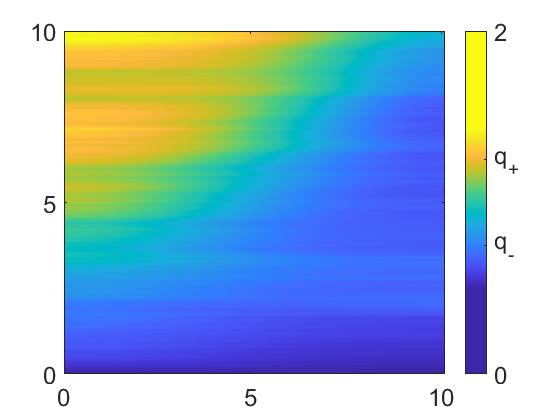}
        \put(-20,380){$t$}
        \put(530,20){$x$}
        \end{overpic}}
        \hspace{8mm}
        \subfloat[Trajectory $q$ solving \eqref{eq:syst_q}\\ with $F= \partial_t$ and $\kappa=0.05$.\\ It shows turbulence onset\\ with respect to the $L^\infty{([0,L])}$-norm.]{\begin{overpic}[scale=0.30]{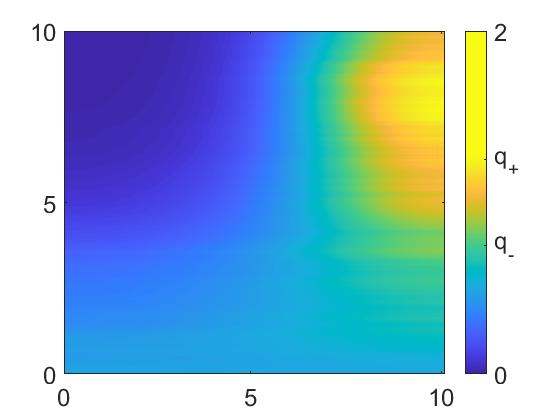}
        \put(-20,380){$t$}
        \put(530,20){$x$}
        \end{overpic}}
        
        \vspace{0mm}
        
        \subfloat[Trajectory $q$ solving \eqref{eq:syst_q}\\ with $F= \partial_t$ and $\kappa=0.5$.\\ It shows turbulence onset\\ with respect to the $L^1{([0,L])}$-norm.]{\begin{overpic}[scale=0.30]{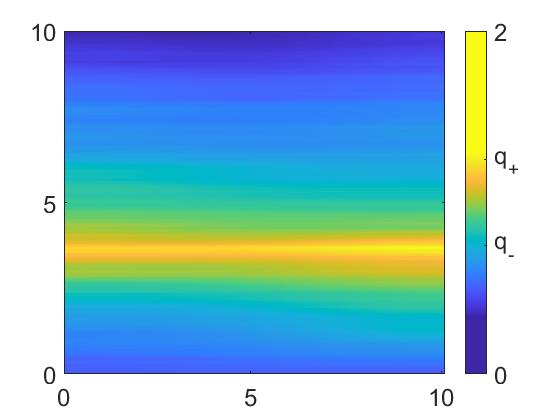}
        \put(-20,380){$t$}
        \put(530,20){$x$}
        \end{overpic}}
        \hspace{8mm}
        \subfloat[Trajectory $q$ solving \eqref{eq:syst_q}\\ with $F= \partial_t$ and $\kappa=0.5$.\\ It shows turbulence onset\\ with respect to the $L^\infty{([0,L])}$-norm.]{\begin{overpic}[scale=0.30]{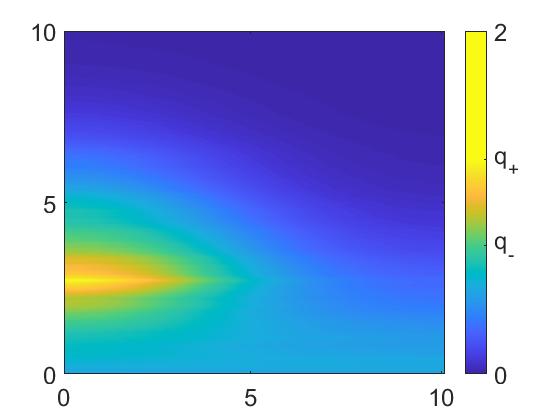}
        \put(-20,380){$t$}
        \put(530,20){$x$}
        \end{overpic}}
        
        \caption{$(a)$ and $(b)$ show trajectories of variable $q$, solution of \eqref{eq:syst_q}, indicating the turbulence onset event under additive red noise (see Appendix \ref{app:A}), $1.5=\sigma_{\text{R}}>\sigma_{\text{I}}=\sigma_{\text{S}}=0$ and $F=\operatorname{Id}$; conversely, in $(c)$, $(d)$, $(e)$ and $(f)$ we consider Stratonovich red noise (see Appendix \ref{app:A}), $0.5=\sigma_{\text{R}}>\sigma_{\text{I}}=\sigma_{\text{S}}=0$ and $F= \partial_t$. Similarly to Figure \ref{fig:Fig1}, the rare events are computed via the TAMS algorithm, for which we obtain $50$ simulations each and use the respective norm as a score function. The size and the discretization of the space-time grid, the operator $Q$, the Reynolds parameter, and the initial condition are chosen as in Figure \ref{fig:Fig1}. We set the perturbation intensity of $\xi$, solution of \eqref{eq:syst_z}, as $\sigma_{\xi}=0.1$ and its dissipation value is indicated under each subfigure. In $(a)$, $(c)$ and $(e)$ we display the rise of $\lvert\lvert q \rvert\rvert_1$ to the value $q_+ L$, whereas in $(b)$, $(d)$ and $(f)$ we capture the rise of $\lvert\lvert q \rvert\rvert_\infty$ to the value $q_+$. The simulations are obtained through the discretized mild solution formula \cite{da2014stochastic}. \\
        The parameter $\kappa$ is associated solely with the dissipation of $\xi$ in the case of additive red noise. This is in contrast with the case $F= \partial_t$, where it also indicates the intensity of a nonlinear perturbation term in \eqref{eq:syst_q}. In $(c)$ and $(d)$, the solution resembles the case of Stratonovich white noise, which corresponds to $\kappa=0$; whereas in $(e)$ and $(f)$, $\kappa$ assumes a higher value and the solution tends to depart from the turbulent state in a short time scale.}
        \label{fig:Fig2}
    \end{figure}
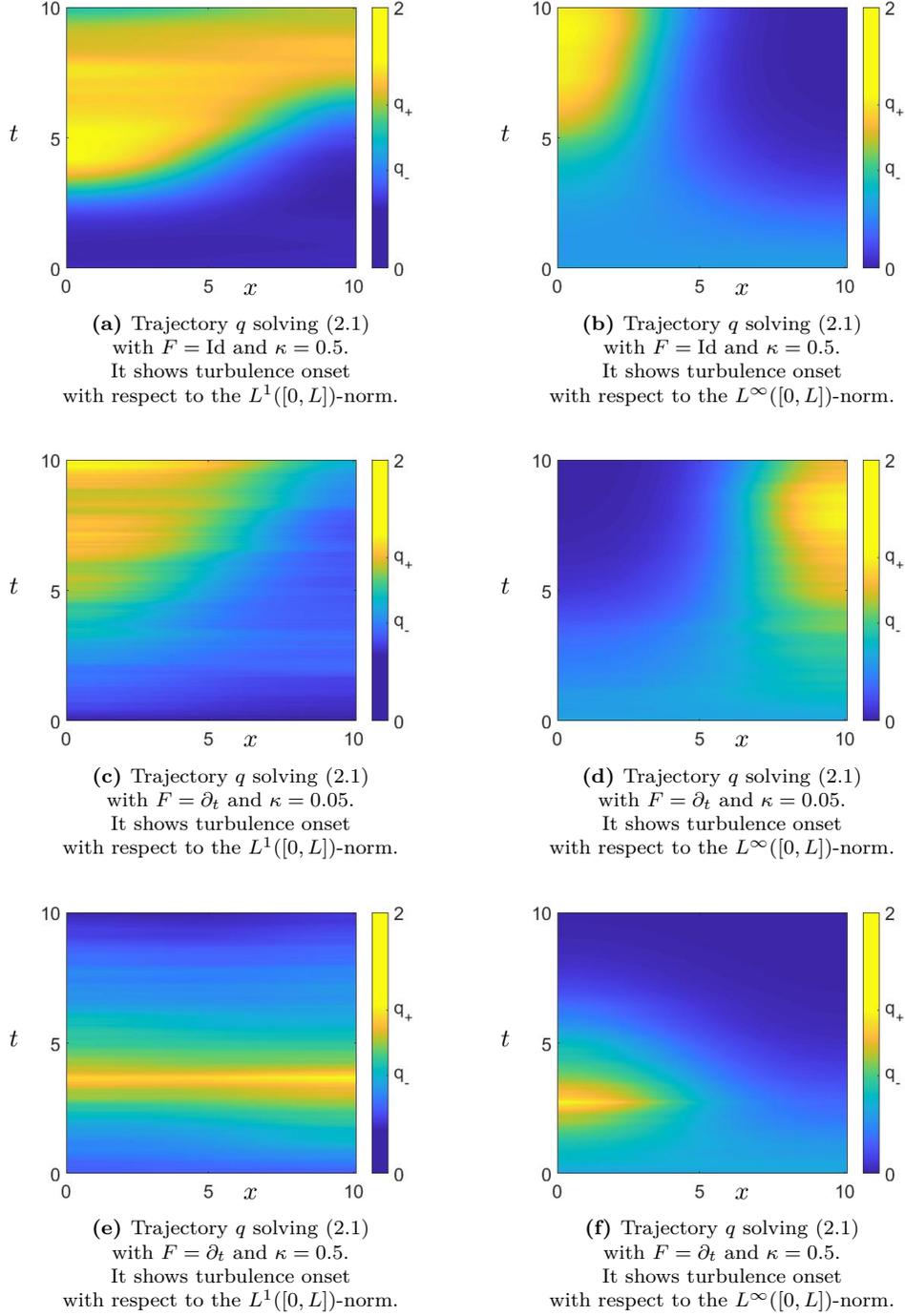 
    
    We indicate with $\mathcal{D}\left(\partial_{xx}^2\right)$ the domain of $\partial_{xx}^2$, which is dense in $L^2([0,L])$ for the assumed boundary conditions. We define the linear operator
    \begin{align*}
        A:= \begin{pmatrix}
            \partial^2_{xx} & \sigma_{\text{R}} \\
            0 & -\kappa
        \end{pmatrix} : \mathcal{D}\left(\partial_{xx}^2\right) \times L^2([0,L]) \to L^2([0,L]) \times L^2([0,L]) ,
    \end{align*}
    and $A^*$ the adjoint operator in respect to $L^2([0,L]) \times L^2([0,L])$. In the mild solution form, the system \eqref{eq:syst_w_SR1} is solved by
    \begin{align*}
        \begin{pmatrix}w^{\text{R}}_g(x,t) \\ \xi(x,t) \end{pmatrix} 
        = \txte^{t A} \begin{pmatrix} \text{log}\left(q_0(x)\right) \\ 0 \end{pmatrix}
        + \int_0^t \txte^{s A} \begin{pmatrix} - g(x) \\ 0 \end{pmatrix} \text{d}s
        + \int_0^t \txte^{(t-s) A} \begin{pmatrix} 0 & 0 \\ 0 & \sigma_{\xi} Q^{\frac{1}{2}} \end{pmatrix} \begin{pmatrix} \text{d}W_s \\ \text{d}W'_s \end{pmatrix} .
    \end{align*}
    The covariance operator 
    \begin{align*}
        V_t:= \begin{pmatrix}
            V^{\RomanNumeralCaps{1}}_t & V^{\RomanNumeralCaps{2}}_t \\ V^{\RomanNumeralCaps{3}}_t & V^{\RomanNumeralCaps{4}}_t
        \end{pmatrix} ,
    \end{align*}
    associated to the solution of the system at time $t>0$, satisfies the finite-time Lyapunov equation \cite[Lemma 2.45]{da2004kolmogorov},
    \begin{align} \label{eq:finite-time-lyap}
        A V_t + V_t A^* = \txte^{t A} \begin{pmatrix} 0 & 0 \\ 0 & \sigma_{\xi}^2 Q \end{pmatrix}  \txte^{t A^*} - \begin{pmatrix} 0 & 0 \\ 0 & \sigma_{\xi}^2 Q \end{pmatrix} .
    \end{align}
    For simplicity, we assume that $-\kappa$ is not an eigenvalue of $\partial^2_{xx}$. The semigroup $\txte^{t A}$ is then defined as
    \begin{align*}
        \txte^{t A}
        =\begin{pmatrix}
            \txte^{t \partial^2_{xx}} & \sigma_{\text{R}} t \int_0^1 \txte^{t s \partial^2_{xx}} \txte^{-t(1-s) \kappa} \text{d}s \\
            0 & \txte^{-t \kappa}
        \end{pmatrix}
        =\begin{pmatrix}
            \txte^{t \partial^2_{xx}} & \sigma_{\text{R}} \operatorname{R}\left( \partial^2_{xx} + \kappa \right) \left( \txte^{t \partial^2_{xx}} - \txte^{-t \kappa} \right) \\
            0 & \txte^{-t \kappa}
        \end{pmatrix} 
    \end{align*}
    and its adjoint, in respect to the $L^2([0,L])\times L^2([0,L])$ scalar product, is
    \begin{align*}
        {\txte^{t A}}^* = \txte^{t A^*}
        =\begin{pmatrix}
            \txte^{t \partial^2_{xx}} & 0 \\
            \sigma_{\text{R}} \operatorname{R}\left( \partial^2_{xx} + \kappa \right) \left( \txte^{t \partial^2_{xx}} - \txte^{-t \kappa} \right) & \txte^{-t \kappa}
        \end{pmatrix} ,
    \end{align*}
    for $\operatorname{R}$ that indicates the resolvent of an operator. Solving \eqref{eq:finite-time-lyap} implies that
    \begin{align*}
        V_t^{\RomanNumeralCaps{2}} &= \sigma_{\text{R}} \sigma_{\xi}^2 \operatorname{R}\left( \partial^2_{xx} - \kappa \right) \left( \txte^{-t \kappa} \operatorname{R}\left( \partial^2_{xx} + \kappa \right) \left( \txte^{t \partial^2_{xx}}-\txte^{-t \kappa} \right) - \frac{1-\txte^{-2 t \kappa}}{2 \kappa}\right) Q , \\
        V_t^{\RomanNumeralCaps{3}} &= \sigma_{\text{R}} \sigma_{\xi}^2 Q \left( \txte^{-t \kappa} \operatorname{R}\left( \partial^2_{xx} + \kappa \right) \left( \txte^{t \partial^2_{xx}}-\txte^{-t \kappa} \right) - \frac{1-\txte^{-2 t \kappa}}{2 \kappa}\right) \operatorname{R}\left( \partial^2_{xx} - \kappa \right) , \\
        V_t^{\RomanNumeralCaps{4}} &= \frac{\sigma_{\xi}^2}{2 \kappa} \left( 1- \txte^{-2 t \kappa} \right) Q ,
    \end{align*}
    and that the equation
    \begin{align} \label{eq:finite-time_lyap_small}
        \partial^2_{xx} V_t^{\RomanNumeralCaps{1}} + V_t^{\RomanNumeralCaps{1}} \partial^2_{xx} 
        = &- \sigma_{\text{R}}^2 \sigma_{\xi}^2 \Bigg( \operatorname{R}\left( \partial^2_{xx} - \kappa \right) \left( \txte^{-t \kappa} \operatorname{R}\left( \partial^2_{xx} + \kappa \right) \left( \txte^{t \partial^2_{xx}} - \txte^{-t \kappa} \right)  - \frac{1-\txte^{-2 t \kappa}}{2 \kappa} \right) Q \nonumber \\
        &+ Q \left( \txte^{-t \kappa} \operatorname{R}\left( \partial^2_{xx} + \kappa \right) \left( \txte^{t \partial^2_{xx}} - \txte^{-t \kappa} \right)  - \frac{1-\txte^{-2 t \kappa}}{2 \kappa} \right) \operatorname{R}\left( \partial^2_{xx} - \kappa \right) \\
        &+ \operatorname{R}\left( \partial^2_{xx} + \kappa \right) \left( \txte^{t \partial^2_{xx}} - \txte^{-t \kappa} \right) 
        Q \operatorname{R}\left( \partial^2_{xx} + \kappa \right) \left( \txte^{t \partial^2_{xx}} - \txte^{-t \kappa} \right) \Bigg) \nonumber
    \end{align}
    holds. For $(n_1,n_2)\in\mathbb{N}\times\mathbb{N}$, we label $p_{n_1,n_2}=\left\langle e_{n_1}, Q e_{n_2} \right\rangle$. For $(n_1,n_2)\neq(0,0)$, the equation \eqref{eq:finite-time_lyap_small} entails the definition of
    \begin{align} \label{eq:definition_gammaS}
            \gamma_{n_1,n_2}
            := &\left\langle e_{n_1}, V^{\RomanNumeralCaps{1}}_t e_{n_2} \right\rangle \nonumber \\
            = &\frac{\sigma_{\text{R}}^2 \sigma_{\xi}^2 p_{n_1,n_2}}{\left( \lambda_{n_1} + \lambda_{n_2} \right)} \Bigg( \frac{1}{\lambda_{n_1} + \kappa} \left( \frac{\txte^{-t \kappa} \left( \txte^{-t \lambda_{n_1}} - \txte^{-t \kappa} \right)}{\lambda_{n_1} - \kappa}  + \frac{1-\txte^{-2 t \kappa}}{2 \kappa} \right) \nonumber \\
            &+ \frac{1}{\lambda_{n_2} + \kappa} \left( \frac{\txte^{-t \kappa} \left( \txte^{-t \lambda_{n_2}} - \txte^{-t \kappa} \right)}{\lambda_{n_2} - \kappa}  + \frac{1-\txte^{-2 t \kappa}}{2 \kappa} \right)
            - \frac{\left(\txte^{-t \lambda_{n_1}}-\txte^{-t \kappa} \right) \left(\txte^{-t \lambda_{n_2}}-\txte^{-t \kappa} \right)}{\left( \lambda_{n_1}-\kappa \right) \left( \lambda_{n_2}-\kappa \right)} \Bigg) \\
            = &\frac{\sigma_{\text{R}}^2 \sigma_{\xi}^2 p_{n_1,n_2}}{\left( \lambda_{n_1} + \lambda_{n_2} \right)} \Bigg( \frac{ \lambda_{n_1} + \lambda_{n_2} + 2 \kappa}{2 \kappa \left(  \lambda_{n_1} + \kappa \right) \left(  \lambda_{n_2} + \kappa \right)}
            - \txte^{-t \left( \lambda_{n_1} + \lambda_{n_2} \right)} \frac{1}{\left( \lambda_{n_1} - \kappa \right) \left( \lambda_{n_2} - \kappa \right)} \nonumber \\
            & +  \txte^{-t \left( \lambda_{n_1} + \kappa \right)} \frac{\lambda_{n_1} + \lambda_{n_2} }{\left( \lambda_{n_1}^2 - \kappa^2 \right) \left( \lambda_{n_2} - \kappa \right)}
            +  \txte^{-t \left( \lambda_{n_2} + \kappa \right)} \frac{\lambda_{n_1} + \lambda_{n_2} }{\left( \lambda_{n_1} - \kappa \right) \left( \lambda_{n_2}^2 - \kappa^2 \right)} 
            - \txte^{-2 t \kappa} \frac{\lambda_{n_1} +\lambda_{n_2}}{2 \kappa \left( \lambda_{n_1} - \kappa \right) \left( \lambda_{n_2} - \kappa \right)}\Bigg) . \nonumber
    \end{align}
    We then consider
    \begin{align*}
        V_t= \sigma_{\xi}^2 \int_0^t \txte^{s A} \begin{pmatrix}
            0 & 0 \\ 0 & Q
        \end{pmatrix} \txte^{s A^*} \text{d} s ,
    \end{align*}
    which implies that
    \begin{align} \label{eq:definition_gamma0}
        \gamma_{0,0}
        &=\left\langle \begin{pmatrix}
            e_0 \\ 0
        \end{pmatrix}, V_t \begin{pmatrix}
            e_0 \\ 0
        \end{pmatrix} \right\rangle_{L^2([0,L])\times L^2([0,L])}
        = \left\langle e_0, V^{\RomanNumeralCaps{1}}_t e_0 \right\rangle \nonumber \\
        &= \sigma_{\text{R}}^2 \sigma_{\xi}^2 \int_0^t \left\langle e_0, \operatorname{R}\left( \partial^2_{xx} + k \right) \left( \txte^{s \partial^2_{xx}} - \txte^{-s \kappa} \right) Q \operatorname{R}\left( \partial^2_{xx} + k \right) \left( \txte^{s \partial^2_{xx}} - \txte^{-s \kappa} \right) e_0 \right\rangle \text{d}s\\
        &= \frac{\sigma_{\text{R}}^2 \sigma_{\xi}^2 p_{0,0}}{\kappa^2} \left( t - 2 \frac{1-\txte^{-t \kappa}}{\kappa} + \frac{1-\txte^{-2 t \kappa}}{2 \kappa}\right) . \nonumber
    \end{align}
    In the theorem to follow, we discuss the lower bound of the probability of rise of $u^{\text{R}}_g$, similarly to Theorem \ref{thm:het_u-I}.
    \begin{thm} \label{thm:het_u-I_2}
        We indicate as $u^{\text{R}}_g$ the strong solution of \eqref{eq:syst_u} for $x\in[0,L]$ and $t>0$. Then the following inequality holds for any non-negative function $f\in L^2([0,L])$ that is not almost everywhere zero:
        \begin{align}  \label{eq:ineq_thm_u-I_2}
            1-\Phi\left( \frac{J'' - \left( \left\langle \txte^{t \partial^2_{xx}} f, \text{log} \left( q_0 \right) \right\rangle - a_0 \left\langle e_0, g \right\rangle t - \overset{\infty}{\underset{n=1}{\sum}} a_n \frac{\left\langle e_n, g \right\rangle}{\lambda_n} \left( 1 - \txte^{-t \lambda_n} \right) \right)}
            {\left( \underset{(n_1,n_2)\in\mathbb{N}^2}{\sum} (a_{n_1} a_{n_2} \gamma_{n_1,n_2} ) \right)^{\frac{1}{2}}} \right)
             \leq \mathbb{P} \left( J'\leq \left\langle f, u^{\text{R}}_g(\cdot,t) \right\rangle \right) ,
        \end{align}
        for $\Phi$, the cumulative distribution function of a standard normal distributed random variable, $J'>0$, $a_n=\left\langle e_n, f \right\rangle$ for any $n\in\mathbb{N}$ and $\left\{ \gamma_{n_1,n_2} \right\}_{(n_1,n_2)\in\mathbb{N}\times\mathbb{N}}$ as defined in \eqref{eq:definition_gammaS} and in \eqref{eq:definition_gamma0}.
    \end{thm}
    \begin{proof}
        The proof follows the same steps as in Theorem \ref{thm:het_u-I}, which we describe here in a more condensed manner. The pairing $\begin{pmatrix}w^{\text{R}}_g(x,t) \\ \xi(x,t) \end{pmatrix}$, solution of \eqref{eq:syst_w_SR1} assumes a Gaussian distribution in $L^2([0,L])\times L^2([0,L])$. In particular, we obtain
        \begin{align*}
            \mathbb{E} \begin{pmatrix}w^{\text{R}}_g(x,t) \\ \xi(x,t) \end{pmatrix} = \txte^{t A} \begin{pmatrix} \text{log}\left(q_0(x)\right) \\ 0 \end{pmatrix}
            + \int_0^t \txte^{s A} \begin{pmatrix} - g(x) \\ 0 \end{pmatrix} \text{d}s
            = \txte^{t \partial^2_{xx}} \text{log}\left(q_0(x)\right)
            - \int_0^t \txte^{s \partial^2_{xx}} g(x) \text{d}s
        \end{align*}
        and its covariance operator is
        \begin{align*}
            \sigma_{\text{R}}^2 \int_0^t \txte^{s A} \begin{pmatrix}
                0 & 0 \\ 0 & Q
            \end{pmatrix} \txte^{s A^*} \text{d}s 
            = \sigma_{\text{R}}^2 \txte^{-2 t} \begin{pmatrix}
                t^2 \int_0^1 \txte^{t s \left( \partial^2_{xx} + \operatorname{Id} \right)} \text{d}s Q \int_0^1 \txte^{t s \left( \partial^2_{xx} + \operatorname{Id} \right)} \text{d}s & t \int_0^1 \txte^{t s \left( \partial^2_{xx} + \operatorname{Id} \right)} \text{d}s Q \\
                t Q \int_0^1 \txte^{t s \left( \partial^2_{xx} + \operatorname{Id} \right)} \text{d}s & Q
            \end{pmatrix} .
        \end{align*}
        Setting $f\in L^2([0,L])$, it follows that
        \begin{align*}
            \mathbb{P}\left( \left\langle f, w^{\text{R}}_g \right\rangle  \geq J'' \right) 
            &= 1-\Phi\left( \frac{J'' - \left\langle f, \txte^{t \partial^2_{xx}} \text{log}\left(q_0\right) - \int_0^t \txte^{s \partial^2_{xx}} g \text{d}s \right\rangle }
            {\left\langle f, V^{\RomanNumeralCaps{1}} f \right\rangle^{\frac{1}{2}}} \right)\\
            &= 1-\Phi\left( \frac{J'' - \left( \left\langle \txte^{t \partial^2_{xx}} f, w^{\text{R}}_g(\cdot,0) \right\rangle - a_0 \left\langle e_0, g \right\rangle t - \overset{\infty}{\underset{n=1}{\sum}} a_n \frac{\left\langle e_n, g \right\rangle}{\lambda_n} \left( 1 - \txte^{-t \lambda_n} \right) \right)}
            {\left( \underset{(n_1,n_2)\in\mathbb{N}^2}{\sum} (a_{n_1} a_{n_2} \gamma_{n_1,n_2} ) \right)^{\frac{1}{2}}} \right) ,
        \end{align*}
        for $J''\in\mathbb{R}$. We assume that $\left\langle f, w^{\text{R}}_g \right\rangle \geq J''$ and employ Jensen's inequality to obtain
        \begin{align*}
             \lvert\lvert f \rvert\rvert_1 \text{exp} \left( \lvert\lvert f \rvert\rvert_1^{-1} J'' \right) 
             &\leq \lvert\lvert f \rvert\rvert_1 \text{exp} \left( \lvert\lvert f \rvert\rvert_1^{-1} \left\langle f, w^{\text{R}}_g \right\rangle \right)
             = \lvert\lvert f \rvert\rvert_1 \text{exp} \left( \lvert\lvert f \rvert\rvert_1^{-1} \int_0^L f(x) w^{\text{R}}_g(x,t) \text{d}x  \right) \\
             &= \lvert\lvert f \rvert\rvert_1 \text{exp} \left( \lvert\lvert f \rvert\rvert_1^{-1} \int_0^L f(x) v^{\text{R}}_g(x,t) \text{d}x  \right)
             \leq \frac{\lvert\lvert f \rvert\rvert_1}{\lvert\lvert f \rvert\rvert_1} \int_0^L f(x)\; \text{exp} (v^{\text{R}}_g(x,t)) \text{d}x \\
             &= \left\langle f, \text{exp} \left( v^{\text{R}}_g(\cdot,t) \right) \right\rangle 
             = \left\langle f, u^{\text{R}}_g(\cdot,t) \right\rangle .
        \end{align*}
        We set $J'=\lvert\lvert f \rvert\rvert_1 \text{exp} \left( \lvert\lvert f \rvert\rvert_1^{-1} J'' \right)$, which implies the statement.
    \end{proof}
    The lower bound \eqref{eq:ineq_thm_u-I_2} implies the following corollary, whose proof is equivalent to those in Corollary \ref{cor:q-I} and Corollary \ref{cor:q-I_sup}. Its statement is discussed in Section \ref{sec:5}.
    
    \begin{cor} \label{cor:q-I_2}
        In the following statements, we refer to $\Phi$ as the cumulative distribution function of a standard normal distributed random variable.
        \begin{enumerate}
            \item[(a)] We assume that $q$, strong solution of \eqref{eq:syst_q}, satisfies $0<q(x,t)\leq 2$ for any $x\in[0,L]$ and $t\in[0,T]$. Then, the following inequality holds for any non-negative function $f\in L^2([0,L])$ that is not almost everywhere zero:
            \begin{align*}
                1-\underset{0\leq t \leq T}{\text{min}}\Phi\left( \frac{\lvert\lvert f \rvert\rvert_1 \text{log} \left( \lvert\lvert f \rvert\rvert_1^{-1} J' \right) - \left\langle \txte^{t \partial^2_{xx}} f, \text{log} \left( q_0 \right) \right\rangle + \left\langle f, e_0 \right\rangle L^{\frac{1}{2}} t}
                {\left( \underset{(n_1,n_2)\in\mathbb{N}^2}{\sum} (a_{n_1} a_{n_2} \gamma_{n_1,n_2} ) \right)^{\frac{1}{2}}} \right)
                 \leq \mathbb{P} \left( J'\leq \underset{0\leq t \leq T}{\text{sup}} \left\langle f, q(\cdot,t) \right\rangle \right) ,
            \end{align*}
            for $J'>0$, $a_n=\left\langle e_n, f \right\rangle$ for any $n\in\mathbb{N}$ and $\left\{ \gamma_{n_1,n_2} \right\}_{(n_1,n_2)\in\mathbb{N}\times\mathbb{N}}$ as defined in \eqref{eq:definition_gammaS} and in \eqref{eq:definition_gamma0}.
            \item[(b)] For $q$, strong solution of \eqref{eq:syst_q} for any $x\in[0,L]$ and $t\in[0,T]$, the following inequality holds
            \begin{align} \label{eq:cor_4.8_b}
                &1-\underset{0\leq t \leq T}{\text{min}}\Phi\left( \frac{L^{\frac{1}{2}} \kappa}{\sigma_{\text{R}} \sigma_{\xi} p_{0,0}^\frac{1}{2}} \frac{t^{\frac{1}{2}}}
                {\left( t - 2 \frac{1-\txte^{-t \kappa}}{\kappa} + \frac{1-\txte^{-2 t \kappa}}{2 \kappa}\right)^{\frac{1}{2}}} \left( t^{-\frac{1}{2}} \left( \text{log}\left( J \right) - L^{-1} \int_0^L \text{log} \left( q_0(x) \right) \text{d}x \right) + t^{\frac{1}{2}} \right) \right) \nonumber\\
                &=1-\underset{0\leq t \leq T}{\text{min}}\Phi\left( \frac{L^{\frac{1}{2}}}{\gamma_{0,0}^{\frac{1}{2}}} \left( \text{log} \left( J \right) - L^{-1} \int_0^L \text{log}\left(q_0(x)\right) \text{d}x + t\right)
                 \right)
                \leq \mathbb{P} \left( J \leq \underset{0\leq t \leq T}{\text{sup}} \lvert\lvert q(\cdot,t) \rvert\rvert_\infty \right) , 
            \end{align}
            for $0<J<2$.
        \end{enumerate}
    \end{cor}

    \begin{remark} \label{rmk:other_red}

    The system \eqref{eq:syst_w_SR1} has a strong solution with variable $w^{\text{R}}_g$ characterized by positive autocorrelation \cite{morr2022red}, conversely to \eqref{eq:syst_w_SR2}. Another core difference in the models is that, for $g\equiv 0$, the covariance of the strong solution along $e_0$ diverges in \eqref{eq:syst_w_SR1} in the time limit, whereas it converges \cite{kuehn2021warning} in \eqref{eq:syst_w_SR2}. Although the models differ greatly, the lower bounds can be obtained through similar methods (see Appendix \ref{app:B}). Furthermore, for $\kappa>0$, the statement in Theorem \ref{thm:het_u-I_2} holds for $w^{\text{R}}_g$ defined in the system \eqref{eq:syst_w_SR2} by setting the constants
        \begin{align} \label{eq:alternative_gammas}
                \gamma_{n_1,n_2}
                = 
                &\frac{\sigma_{\text{R}}^2 \sigma_{\xi}^2 p_{n_1,n_2}}{\left( \lambda_{n_1} + \lambda_{n_2} \right)} \Bigg( \frac{2 \lambda_{n_1} \lambda_{n_2} + \left( \lambda_{n_1} + \lambda_{n_2} \right) \kappa}{2 \left(  \lambda_{n_1} + \kappa \right) \left(  \lambda_{n_2} + \kappa \right)}
                - \txte^{-t \left( \lambda_{n_1} + \lambda_{n_2} \right)} \frac{ \lambda_{n_1} \lambda_{n_2}}{\left( \lambda_{n_1} - \kappa \right) \left( \lambda_{n_2} - \kappa \right)} \nonumber \\
                & +  \txte^{-t \left( \lambda_{n_1} + \kappa \right)} \frac{\kappa \lambda_{n_1} \left( \lambda_{n_1} + \lambda_{n_2} \right)}{\left( \lambda_{n_1}^2 - \kappa^2 \right) \left( \lambda_{n_2} - \kappa \right)}
                +  \txte^{-t \left( \lambda_{n_2} + \kappa \right)} \frac{\kappa \lambda_{n_2} \left( \lambda_{n_1} + \lambda_{n_2} \right)}{\left( \lambda_{n_1} - \kappa \right) \left( \lambda_{n_2}^2 - \kappa^2 \right)} \\
                & - \txte^{-2 t \kappa} \frac{\kappa \left(\lambda_{n_1} +\lambda_{n_2}\right)}{2 \left( \lambda_{n_1} - \kappa \right) \left( \lambda_{n_2} - \kappa \right)}\Bigg) , \nonumber
        \end{align}
    for $(n_1,n_2)\in \mathbb{N}\times\mathbb{N}\setminus (0,0)$, and
        \begin{align} \label{eq:alternative_gammas_0}
            &\gamma_{0,0}
            = \sigma_{\text{R}}^2 \sigma_{\xi}^2 p_{0,0} \frac{1-\txte^{-2 t \kappa}}{2 \kappa} ,
        \end{align}
    and assuming, for simplicity, that $-\kappa$ is not an eigenvalue of the Laplacian. For system assumptions associated to \eqref{eq:syst_w_SR2}, the statement in Corollary \ref{cor:q-I_2} $(a)$ is also equivalent, with the updated constants $\left\{\gamma_{n_1,n_2}\right\}_{(n_1,n_2)\in\mathbb{N}\times\mathbb{N}}$ shown in \eqref{eq:alternative_gammas} and \eqref{eq:alternative_gammas_0}. In contrast, the corresponding inequality to \eqref{eq:cor_4.8_b} is
        \begin{align} \label{eq:RS_last}
                &1-\underset{0\leq t \leq T}{\text{min}}\Phi\left( \frac{(2 \kappa L)^{\frac{1}{2}}}{\sigma_{\text{R}} \sigma_{\xi} p_{0,0}^\frac{1}{2}} \frac{t^{\frac{1}{2}}}
                {\left( 1-\txte^{-2 t \kappa} \right)^{\frac{1}{2}}} \left( t^{-\frac{1}{2}} \left( \text{log}\left( J \right) - L^{-1} \int_0^L \text{log} \left( q_0(x) \right) \text{d}x \right) + t^{\frac{1}{2}} \right) \right)\\
                &=1-\underset{0\leq t \leq T}{\text{min}}\Phi\left( \frac{L^{\frac{1}{2}}}{\gamma_{0,0}^{\frac{1}{2}}} \left( \text{log} \left( J \right) - L^{-1} \int_0^L \text{log}\left(q_0(x)\right) \text{d}x + t\right)
                 \right)
                \leq \mathbb{P} \left( J \leq \underset{0\leq t \leq T}{\text{sup}} \lvert\lvert q(\cdot,t) \rvert\rvert_\infty \right) , \nonumber
        \end{align}
    for $0<J<2$ and $\gamma_{0,0}$ defined in \eqref{eq:alternative_gammas_0}.
    \end{remark}
    
\section{Comparison of methods} \label{sec:5}
    
In the previous sections, we have studied lower bounds to the probability of turbulence onset in system \eqref{eq:syst_q} under different noise assumptions. Namely, we considered time-white noise in It\^o sense, in Section \ref{sec:3}, and time-white or time-red noise in Stratonovich sense, in Section \ref{sec:4}. Throughout the paper, the noise is assumed to be white in space along fixed modes. The techniques employed in the two sections differ in nature yet manage to capture similar characteristics of the system.
\smallskip

We first discuss the rise of turbulence under It\^o noise. Lemma \ref{lm:inequality_quadratic_variation} and Lemma \ref{lm:galton-watson} were based on methods introduced in \cite{mueller2000critical,mueller1993blowup}. The key idea is to counter the drift effect on the solution of \eqref{eq:syst_gamma} by applying a convolution operator through the fundamental solution of the cable equation, reserved in time. A stark difference with the assumptions in \cite{mueller2000critical,mueller1993blowup} is that their work studies the heat equation, i.e., $\alpha=0$. In fact, under strong linear drift dissipation, the left-hand side term in \eqref{eq:in_nom_infty_1} can be negative, even for large values of $T$, thus rendering the inequality trivial. This is in contrast with \cite{mueller2000critical}, where the lower bound increases with $T$. In \cite{mueller2000critical,mueller1993blowup}, a second lower bound is obtained and employed to prove the blow-up of the mild solution in finite time through an iterative method. The iteration step is based on rescaling the space interval. Once again, the method is not equivalent under the assumption $\alpha>0$ as the second rescaling affects the linear dissipation and does not return the initial system, thus affecting the iteration at each step (see \cite[Lemma 2.6, Proposition 3.2, Section 4]{mueller2000critical}).
\smallskip

We note that the lower bounds in Corollary \ref{cor:mueller} are also valid for $q$ mild solution of \eqref{eq:syst_q} under white Stratonovich noise. In fact, we set $u^{\text{S}}_1=u^{\text{S}}_1(x,t)$ the mild solution of \eqref{eq:syst_gamma} for $\hat{\sigma}=\sigma_{\text{S}}>\sigma_{\text{I}}=\sigma_{\text{R}}=0$ and $u^{\text{I}}_1=u^{\text{I}}_1(x,t)$ the mild solution of \eqref{eq:syst_gamma} for $\hat{\sigma}=\sigma_{\text{I}}>\sigma_{\text{S}}=\sigma_{\text{R}}=0$. We then study $\hat{u}=u^{\text{S}}_1-u^{\text{I}}_1$. By assuming $Q$ to be trace-class and by the definition of the It\^o-Stratonovich correction term in \cite{barna2019analytic,gardiner1985handbook}, it is the mild solution of 
\begin{align*}
    \left\{
    \begin{alignedat}{2}
    \text{d} \hat{u}(x,t) &= \left( \partial_{xx}^2 \hat{u}(x,t)- \hat{u}(x,t) + u^{\text{I}}(x,t) \frac{\hat{\sigma}^2 }{2} \sum_{i=0}^\infty \zeta_i b_i(x)^2 \right) \text{d}t + \hat{\sigma} \hat{u}(x,t) Q^{\frac{1}{2}} \text{d}W_t , \\
    \partial_x \hat{u}(0,t) &= \partial_x \hat{u}(L,t) = 0 , \\
    \hat{u}(x,0) &\equiv 0 .
    \end{alignedat}
    \right.
\end{align*}
Following the steps in the proof of Lemma \ref{lm:q-u}, it is implied that $u^{\text{S}}_1\geq u^{\text{I}}_1$ almost surely. Moreover, we know that $\upsilon^{\text{S}}_{J} \leq \upsilon^{\text{I}}_{J}$ holds.
\smallskip

The methods in Section \ref{sec:3} are applied to the mild solutions of the considered stochastic partial differential equations. Conversely, although in Section \ref{sec:4} the noise is assumed in Stratonovich sense, the use of It\^o's Lemma through the inverse Cole-Hopf transformation requires the study of strong solutions. In order to satisfy the existence of such a solution, the noise is assumed to perturb the system along a finite number of modes, then affected by the multiplication term in the noise intensity \cite[Section 6.5, Chapter 7]{da2014stochastic}. Similar results to those in Section \ref{sec:4} can also be proven in an equivalent manner under It\^o noise assumptions, yet, on a logarithmic scale, an additional negative term appears in the KPZ equation \cite{hairer2013solving}. For instance, under the assumption of white noise, the system obtained in correspondence to \eqref{eq:syst_v} is
\begin{align*}
    \left\{
    \begin{alignedat}{2}
    \text{d} v^\text{I}_{g}(x,t) &= \left( \partial_{xx}^2 v^\text{I}_{g}(x,t) + \left(\partial_x v^\text{I}_{g}(x,t) \right)^2 - g(x) - \frac{\sigma_{\text{I}}^2}{2} \sum_{i=0}^m \zeta_i b_i(x)^2 \right) \text{d}t + \sigma_{\text{I}} Q^{\frac{1}{2}} \text{d}W_t , \\
    \partial_x v^\text{I}_{g}(0,t) &= \partial_x v^\text{I}_{g}(L,t) = 0 , \\
    v^\text{I}_{g}(x,0) &= \text{log}\left(q_0(x)\right) .
    \end{alignedat}
    \right.
\end{align*}
Then, the step in Lemma \ref{lm:v-w} implies the discard of the nonlinear shear deformation non-negative term, but not of the term $- \frac{\sigma_{\text{I}}^2}{2} \sum_{i=0}^m \zeta_i b_i(x)^2$, which has to be included in the moving average of Theorem \ref{thm:het_u-I}, Corollary \ref{cor:q-I} and Corollary \ref{cor:q-I_sup}.
\smallskip

Although the lower bounds in Section \ref{sec:4} are always positive, they are characterized by severely different behaviours in time. The section contains results that manage to study the initiation of turbulence in certain regions of the domain in Corollary \ref{cor:q-I} and Corollary \ref{cor:q-I_2}. Nevertheless, we focus on the description of jumps in the $L^\infty([0,L])$-norm, which are easier to observe. In Corollary \ref{cor:q-I_sup}, the sum of two terms in the argument of the function $\Phi$ defines the nature of the bound. In fact, the term $t^{-\frac{1}{2}}$, multiplied by a positive parameter dependent on initial conditions, refers to the effect of the noise to enable the possibility of a jump, thus decreasing with $t$ the argument in the distribution function. Conversely, the term $t^{\frac{1}{2}}$ takes into account the effect of the drift component, constant in time, which hinders the transition on long times. The clash of these effects implies that at
\begin{align*}
    t^*=\text{log}\left( J \right) - L^{-1} \int_0^L \text{log} \left( q_0(x) \right) \text{d}x
\end{align*}
the lower bound to the probability of jump at height $J$ in the $L^1([0,L])$-norm reaches its peak, which is maintained for longer times for the statement in the proof of Corollary \ref{cor:mueller}. 
A similar behaviour is observed in Corollary \ref{cor:mueller}, where the best opportunity of occurrence of jump is captured.
\smallskip

We considered in Subsection \ref{subsec:red} red noise in time. We note that, in previous cases, the initiation of turbulence is not a memoryless process since it is already affected by its initial state. The inclusion of red noise in the sense \eqref{eq:syst_w_SR1} is justified by its relevance in climate application \cite{morr2024detection,morr2022red} and the positive autocorrelation of the system in logarithmic scale along the direction $e_0$. The inequality \eqref{eq:cor_4.8_b} in Corollary \ref{cor:q-I_2} $(b)$ resembles the statement in Corollary \ref{cor:q-I_sup}, but the parameter $\kappa$ assumes a twofold role: it is proportional to the term
\begin{align} \label{eq:red-term_0}
    \frac{L^{\frac{1}{2}} \kappa}{\sigma_{\text{R}} \sigma_{\xi} p_{0,0}^\frac{1}{2}}
\end{align}
and appears in the term
\begin{align} \label{eq:red-term}
    \frac{t^{\frac{1}{2}}}{\left( t - 2 \frac{1-\txte^{-t \kappa}}{\kappa} + \frac{1-\txte^{-2 t \kappa}}{2 \kappa}\right)^{\frac{1}{2}}} 
    = \left(\frac{t}{\int_0^t \left(1-\txte^{-s \kappa} \right)^2 \text{d} s}\right)^{\frac{1}{2}} > 1 ,
\end{align}
due to its dissipative role in \eqref{eq:syst_z}. Both perspectives indicate that $\kappa>0$ hinders the jump with respect to the white noise assumption, but their effects differ. In fact, increasing $\kappa$ leads to an increase in \eqref{eq:red-term_0} and a decrease in \eqref{eq:red-term}. Still, the concept of a best opportunity to jump before being fully controlled by the drift is visible since, for fixed $\kappa>0$, the limit of \eqref{eq:red-term} in $t\to+\infty$ is finite. Nevertheless, the limit of \eqref{eq:red-term} in $t\to 0$ is infinite, indicating a smaller probability of jump on short times in comparison to white noise in time. Lastly, we notice that the limit $\kappa=0$ implies $\gamma_{0,0}= \sigma_{\text{R}}^2 \sigma_{\xi}^2 p_{0,0} \frac{t^3}{3}$, which means that the lower bound is a definitely strictly increasing function in $t$.
\smallskip

At the end of Subsection \ref{subsec:red}, we consider an alternative type of red noise, as discussed in \cite{kuehn2021warning}. The role of $\kappa$ is different from the previous case. In fact, $\kappa$ indicates both the dissipativity in \eqref{eq:syst_z} and is related to the off-diagonal entry in the linear drift term (see Appendix \ref{app:B}). The term in \eqref{eq:RS_last},
\begin{align} \label{eq:red-term_2}
    \frac{t^{\frac{1}{2}}}{\left( \frac{1-\txte^{-2 t \kappa}}{2 \kappa}\right)^{\frac{1}{2}}} 
    = \left(\frac{t}{\int_0^t \txte^{-2 s \kappa} \text{d} s}\right)^{\frac{1}{2}}
\end{align}
is finite in $t\to 0$ and infinite in $t\to +\infty$. This translates to a similar behaviour to the lower bound with the assumption of white noise in Corollary \ref{cor:q-I} on small time scales and a smaller lower bound to the probability of jump on large time scales $t$. This results from the interplay between the dissipativity in the Ornstein-Uhlenbeck process defined in \eqref{eq:syst_z} and the dissipativity in \eqref{eq:syst_u}. In this case, increasing $\kappa>0$ indicates in \eqref{eq:red-term_2} the decrease in the bound of jump at a fixed time. The limit $\kappa=0$ is equivalent to the time-white noise assumption in Subsection \ref{subsec:white}.

\section*{Conclusion}

In this paper, we have established lower bounds for the probability of turbulence initiation in a simplified model of plane Couette flow. The phenomenon is modeled through an SPDE on an interval under various forms of multiplicative Gaussian noise. The simulation of similar systems leads to the study of other rare events in the pipe flow \cite{gome2024phase,gome2022extreme}. The application of equivalent algorithms shows the possibility of such an occurrence, whose probability has not been estimated analytically so far. Central to our approach is the comparison between the linearized and original models, starting with a rigorous treatment of It\^o white noise. Using the fundamental solution of the cable equation under Neumann or periodic boundary conditions, we counter the drift term and obtain an observable in the form of a martingale. The bound is then described following techniques employed in the study of finite-time blow-up of the mild solution of the stochastic heat equation \cite{mueller2000critical,mueller1993blowup}.
\smallskip

Expanding this framework, we employ logarithmic-scale comparison methods to address systems perturbed by time-white Stratonovich noise, time-red additive noise, and time-red Stratonovich noise, with the latter two modeled through coupling with Ornstein–Uhlenbeck processes. This approach requires different assumptions on the noise term to ensure the existence of a strong solution for the model. Our analysis reveals that, whereas the lower bounds vary across noise types, they consistently reflect a key dynamical feature: the probability of turbulence onset in a fixed time declines beyond an optimal jumping time. This property is implied by the invariance of the linearized model under a suitable rescaling of the solution.
\smallskip

Finally, we illustrate the adaptability of our techniques to other types of noise, highlighting their potential for broader applicability in turbulence modeling. These findings provide a deeper understanding of stochastic influences on turbulence initiation and offer ground for future investigations into noise-driven dynamics in SPDEs perturbed by multiplicative Gaussian noise.

\section*{Acknowledgments}
Author PB wants to thank Andreas Morr for his valuable insights on red noise and climate science.

\newpage
\bibliographystyle{abbrv}
\bibliography{Ref}

\newpage
\appendix

\section{Appendix: Turbulence from the It\^o noise perspective}
    \label{app:A}
    Throughout the paper, system \eqref{eq:syst_q} is studied under different noise and parameters assumptions. In this appendix, we address the various cases and translate them to the It\^o noise perspective through the It\^o-Stratonovich correction term \cite{gardiner1985handbook,twardowska2004relation}. The boundary conditions are assumed to be homogenous Neumann, and the initial condition is set at $q(x,0)=q_0(x)$.
    \begin{itemize}
        \item   In Section \ref{sec:3}, the noise is assumed to be white and in It\^o sense, i.e., $\sigma_{\text{S}} = \sigma_{\text{R}}=0$. The system studied is characterized, therefore, by the form
        \begin{align*}
            \text{d} q(x,t) = \left( \partial_{xx}^2 q(x,t) - q(x,t) + (r+1) q(x,t)^2 (2-q(x,t)) \right) \text{d}t + \sigma_{\text{I}} q(x,t) Q^{\frac{1}{2}} \text{d}W_t ,
        \end{align*}
        for $x\in[0,L]$ and $t>0$.
        \item   In Subsection \ref{subsec:white}, the noise is assumed white and in Stratonovich sense, i.e., $\sigma_{\text{I}} = \sigma_{\text{R}}=0$. The first equation in \eqref{eq:syst_q} with It\^o noise is then
        \begin{align*}
            \text{d} q(x,t) = \left( \partial_{xx}^2 q(x,t) - q(x,t) + (r+1) q(x,t)^2 (2-q(x,t)) + \frac{\sigma_{\text{S}}^2}{2} q(x,t) \sum_{i=0}^m \left( \zeta_i b_i(x)^2 \right) \right) \text{d}t + \sigma_{\text{S}} q(x,t) Q^{\frac{1}{2}} \text{d}W_t .
        \end{align*}
        \item   In Subsection \ref{subsec:red}, we discuss the effect of red in time Stratonovich noise on the turbulence system, i.e., $\sigma_{\text{R}}>\sigma_{\text{I}} =\sigma_{\text{S}}=0$. The variable $q$ is coupled to the Ornstein-Uhlenbeck process, solution of \eqref{eq:syst_z}, depending on the choice of the operator $F$. In \eqref{eq:syst_w_SR1}, we set $F=\operatorname{Id}$. Then, the first equation in \eqref{eq:syst_q} coupled with \eqref{eq:syst_z} can be interpreted as
        \begin{align*}
            \text{d}\begin{pmatrix}
                q(x,t) \\ \xi(x,t)
            \end{pmatrix} 
            = \begin{pmatrix}
                \partial_{xx}^2 q(x,t) - q(x,t) + (r+1) q(x,t)^2 (2-q(x,t)) + \sigma_{\text{R}} q(x,t) \xi(x,t) \\
                - \kappa \xi
            \end{pmatrix} \text{d}t
            + \begin{pmatrix} 0 & 0 \\ 0 & \sigma_{\xi} Q^{\frac{1}{2}} \end{pmatrix} \begin{pmatrix} \text{d}W_t \\ \text{d}W'_t \end{pmatrix} .
        \end{align*}
        The noise is additive in $L^2([0,L])\times L^2([0,L])$, which implies that the It\^o-Stratonovich correction term is null and the equation is equivalent regardless of the interpretation of the noise.    
        \item   In Remark \ref{rmk:other_red}, we discuss the lower bound to the probability of initiation of turbulence under Stratonovich red noise defined by $F= \partial_t$. The equation defining the behaviour of the variables $q$ and $\xi$ in time is, therefore,
        \begin{align*}
            \text{d}\begin{pmatrix}
                q(x,t) \\ \xi(x,t)
            \end{pmatrix} 
            = &\begin{pmatrix}
                \partial_{xx}^2 q(x,t) - q(x,t) + (r+1) q(x,t)^2 (2-q(x,t)) \\
                - \kappa \xi
            \end{pmatrix} \text{d}t
            + \begin{pmatrix}
                \sigma_{\text{R}} q(x,t) \\ 0
            \end{pmatrix} \circ \text{d}\xi(x,t)
            + \begin{pmatrix} 0 & 0 \\ 0 & \sigma_{\xi} Q^{\frac{1}{2}} \end{pmatrix} \begin{pmatrix} \text{d}W_t \\ \text{d}W'_t \end{pmatrix} \\
            = &\begin{pmatrix}
                \partial_{xx}^2 q(x,t) - q(x,t) + (r+1) q(x,t)^2 (2-q(x,t)) - \sigma_{\text{R}} \kappa q(x,t) \xi(x,t) \\
                - \kappa \xi
            \end{pmatrix} \text{d}t \\
            &+ \begin{pmatrix} q(x,t) & 0 \\ 0 & 1 \end{pmatrix} \circ 
            \begin{pmatrix} 0 & \sigma_{\text{R}} \sigma_{\xi} Q^{\frac{1}{2}} \\ 0 & \sigma_{\xi} Q^{\frac{1}{2}} \end{pmatrix} \begin{pmatrix} \text{d}W_t \\ \text{d}W'_t \end{pmatrix} ,
        \end{align*}
        with noise in Stratonovich sense. Indicating with $*$ the adjoint operator in $L^2([0,L])\times L^2([0,L])$, the operator 
        \begin{align*}
            \begin{pmatrix} 0 & \sigma_{\text{R}} \sigma_{\xi} Q^{\frac{1}{2}} \\ 0 & \sigma_{\xi} Q^{\frac{1}{2}} \end{pmatrix} 
            \begin{pmatrix} 0 & \sigma_{\text{R}} \sigma_{\xi} Q^{\frac{1}{2}} \\ 0 & \sigma_{\xi} Q^{\frac{1}{2}} \end{pmatrix}^*
            = \sigma_{\xi}^2 \begin{pmatrix} 0 & \sigma_{\text{R}} Q^{\frac{1}{2}} \\ 0 & Q^{\frac{1}{2}} \end{pmatrix} 
            \begin{pmatrix} 0 & 0 \\ \sigma_{\text{R}} Q^{\frac{1}{2}} & Q^{\frac{1}{2}} \end{pmatrix}
            = \sigma_{\xi}^2 \begin{pmatrix} \sigma_{\text{R}}^2 Q & \sigma_{\text{R}} Q \\ \sigma_{\text{R}} Q & Q \end{pmatrix}
        \end{align*}
        is characterized by purely discrete spectrum, composed by $\left\{ \text{\Coppa}_i,\text{\Coppa}_i' \right\}_{i\in\mathbb{N}}$. These eigenvalues are defined as
        \begin{align*}
            \text{\Coppa}_i := \left( 1 + \sigma_{\text{R}}^2 \right) \sigma_{\xi}^2 \zeta_i 
            \qquad \text{and} \qquad
            \text{\Coppa}_i' := 0 ,
        \end{align*}
        for all $i\in\mathbb{N}$. The corresponding eigenbasis in $L^2([0,L])\times L^2([0,L])$ is composed by
        \begin{align*}
            B_i(x) := \left( 1 + \sigma_{\text{R}}^2 \right)^{-\frac{1}{2}} 
            \begin{pmatrix}
                \sigma_{\text{R}} b_i(x) \\ b_i(x)
            \end{pmatrix} 
            \qquad \text{and} \qquad
            B_i'(x) := \left( 1 + \sigma_{\text{R}}^2 \right)^{-\frac{1}{2}} 
            \begin{pmatrix}
                b_i(x) \\ - \sigma_{\text{R}} b_i(x)
            \end{pmatrix}  ,
        \end{align*}
        for all $i\in\mathbb{N}$ and $x\in[0,L]$. The equation defining $q$ and $\xi$ is, then,
        \begin{align*}
            \text{d}\begin{pmatrix}
                q(x,t) \\ \xi(x,t)
            \end{pmatrix}
            = &\begin{pmatrix}
                \partial_{xx}^2 q(x,t) - q(x,t) + (r+1) q(x,t)^2 (2-q(x,t)) - \sigma_{\text{R}} \kappa q(x,t) \xi(x,t) \\
                - \kappa \xi
            \end{pmatrix} \text{d}t \\
            & + \frac{1}{2  \left( 1 + \sigma_{\text{R}}^2 \right) } \begin{pmatrix}
                 q(x,t) \partial_q\left(q(x,t)\right) \overset{m}{\underset{i=0}{\sum}} \left( \text{\Coppa}_i \sigma_{\text{R}}^2 b_i(x)^2 + \text{\Coppa}_i' b_i(x)^2 \right) \\ 
                 \partial_{\xi}\left( 1 \right) \overset{m}{\underset{i=0}{\sum}} \left( \text{\Coppa}_i b_i(x)^2 + \text{\Coppa}_i' \sigma_{\text{R}}^2 b_i(x)^2 \right)
            \end{pmatrix}
            + \begin{pmatrix} q(x,t) & 0 \\ 0 & 1 \end{pmatrix} \begin{pmatrix} 0 & \sigma_{\text{R}} \sigma_{\xi} Q^{\frac{1}{2}} \\ 0 & \sigma_{\xi} Q^{\frac{1}{2}} \end{pmatrix} \begin{pmatrix} \text{d}W_t \\ \text{d}W'_t \end{pmatrix} \\
            = &\begin{pmatrix}
                \partial_{xx}^2 q(x,t) - q(x,t) + (r+1) q(x,t)^2 (2-q(x,t)) - \sigma_{\text{R}} \kappa q(x,t) \xi(x,t) + \frac{\sigma_{\text{R}}^2 \sigma_{\xi}^2}{2} q(x,t) \overset{m}{\underset{i=0}{\sum}} \left( \zeta_i b_i(x)^2 \right) \\
                - \kappa \xi
            \end{pmatrix} \text{d}t \\
            &+ \begin{pmatrix} q(x,t) & 0 \\ 0 & 1 \end{pmatrix} \begin{pmatrix} 0 & \sigma_{\text{R}} \sigma_{\xi} Q^{\frac{1}{2}} \\ 0 & \sigma_{\xi} Q^{\frac{1}{2}} \end{pmatrix} \begin{pmatrix} \text{d}W_t \\ \text{d}W'_t \end{pmatrix} ,
        \end{align*}
        with noise interpreted in It\^o sense.
    \end{itemize}

\section{Appendix: Proof of lower bounds for alternative red noise} 
    \label{app:B}
    In this appendix, we study the covariance of the $w^{\text{R}}_g=w^{\text{R}}_g(x,t)$ component of the strong solution of \eqref{eq:syst_w_SR2}. The method resembles the one employed on the strong solution \eqref{eq:syst_w_SR1} in Section \ref{sec:4}, which we outline below concisely. We define the operator
    \begin{align*}
        A:= \begin{pmatrix}
            \partial^2_{xx} & -\kappa \sigma_{\text{R}} \\
            0 & -\kappa
        \end{pmatrix}: \mathcal{D}\left(\partial_{xx}^2\right) \times L^2([0,L]) \to L^2([0,L]) \times L^2([0,L]) 
    \end{align*}
    and its adjoint in $L^2([0,L])\times L^2([0,L])$,
    \begin{align*}
        A^*:= \begin{pmatrix}
            \partial^2_{xx} & 0 \\
            -\kappa \sigma_{\text{R}} & -\kappa
        \end{pmatrix}: \mathcal{D}\left(\partial_{xx}^2\right) \times L^2([0,L]) \to L^2([0,L]) \times L^2([0,L]) .
    \end{align*}
    The covariance operator, 
    \begin{align*}
        V_t:= \begin{pmatrix}
            V^{\RomanNumeralCaps{1}}_t & V^{\RomanNumeralCaps{2}}_t \\ V^{\RomanNumeralCaps{3}}_t & V^{\RomanNumeralCaps{4}}_t
        \end{pmatrix} ,
    \end{align*}
    of the solution of the system at time $t>0$ satisfies the finite-time Lyapunov equation,
    \begin{align} \label{eq:finite-time-lyap_2}
        A V_t + V_t A^* = \txte^{t A} \begin{pmatrix} \sigma_{\text{R}}^2 \sigma_{\xi}^2 Q & \sigma_{\text{R}} \sigma_{\xi}^2 Q \\ \sigma_{\text{R}} \sigma_{\xi}^2 Q & \sigma_{\xi}^2 Q \end{pmatrix}  \txte^{t A^*} - \begin{pmatrix} \sigma_{\text{R}}^2 \sigma_{\xi}^2 Q & \sigma_{\text{R}} \sigma_{\xi}^2 Q \\ \sigma_{\text{R}} \sigma_{\xi}^2 Q & \sigma_{\xi}^2 Q \end{pmatrix} .
    \end{align}
    Once again, for simplicity, we assume that $-\kappa$ is not an eigenvalue of $\partial^2_{xx}$. The semigroup $\txte^{t A}$ is then defined as
    \begin{align*}
        \txte^{t A}
        =\begin{pmatrix}
            \txte^{t \partial^2_{xx}} & - \kappa \sigma_{\text{R}} t \int_0^1 \txte^{t s \partial^2_{xx}} \txte^{-t (1-s) \kappa} \text{d}s \\
            0 & \txte^{-t \kappa}
        \end{pmatrix}
        =\begin{pmatrix}
            \txte^{t \partial^2_{xx}} & - \kappa \sigma_{\text{R}} \operatorname{R}\left( \partial^2_{xx} + \kappa \right) \left( \txte^{t \partial^2_{xx}} - \txte^{-t \kappa} \right) \\
            0 & \txte^{-t \kappa}
        \end{pmatrix} ,
    \end{align*}
    and its adjoint, in respect to the $L^2([0,L])\times L^2([0,L])$ scalar product, as
    \begin{align*}
        {\txte^{t A}}^* = \txte^{t A^*}
        =\begin{pmatrix}
            \txte^{t \partial^2_{xx}} & 0 \\
            - \kappa \sigma_{\text{R}} \operatorname{R}\left( \partial^2_{xx} + \kappa \right) \left( \txte^{t \partial^2_{xx}} - \txte^{-t \kappa} \right) & \txte^{-t \kappa}
        \end{pmatrix} .
    \end{align*}
    Solving \eqref{eq:finite-time-lyap_2} implies that
    \begin{align*}
        V_t^{\RomanNumeralCaps{2}} &= \sigma_{\text{R}} \sigma_{\xi}^2 \operatorname{R} \left(\partial^2_{xx} - \kappa \right) \left( - \frac{1}{2} \left( 1 + \txte^{-2 t \kappa} \right) + \txte^{t \left(\partial^2_{xx}-\kappa\right)} - \kappa \operatorname{R}\left( \partial^2_{xx} + \kappa \right) \txte^{-t \kappa} \left( \txte^{t \partial^2_{xx}} - \txte^{-t \kappa} \right) \right) Q , \\
        V_t^{\RomanNumeralCaps{3}} &= \sigma_{\text{R}} \sigma_{\xi}^2 Q \left( - \frac{1}{2} \left( 1 + \txte^{-2 t \kappa} \right) + \txte^{t \left( \partial^2_{xx} - \kappa \right)} - \kappa \operatorname{R}\left( \partial^2_{xx} + \kappa \right) \txte^{-t \kappa} \left( \txte^{t \partial^2_{xx}} - \txte^{-t \kappa} \right) \right) \operatorname{R}\left( \partial^2_{xx} - \kappa \right) , \\
        V_t^{\RomanNumeralCaps{4}} &= \frac{\sigma_{\xi}^2}{2 \kappa} \left( 1- \txte^{-2 t \kappa} \right) Q ,
    \end{align*}
    and that the equation
    \begin{align} \label{eq:finite-time_lyap_small_2}
        \partial^2_{xx} V_t^{\RomanNumeralCaps{1}} + V_t^{\RomanNumeralCaps{1}} \partial^2_{xx} 
        = &\sigma_{\text{R}}^2 \sigma_{\xi}^2 \Bigg( \kappa \operatorname{R} \left(\partial^2_{xx} - \kappa \right) \left( - \frac{1}{2} \left( 1 + \txte^{-2 t \kappa} \right) + \txte^{t \left(\partial^2_{xx}-\kappa\right)} - \kappa \operatorname{R}\left( \partial^2_{xx} + \kappa \right) \txte^{-t \kappa} \left( \txte^{t \partial^2_{xx}} - \txte^{-t \kappa} \right) \right) Q \nonumber \\
        &+ Q \left( - \frac{1}{2} \left( 1 + \txte^{-2 t \kappa} \right) + \txte^{t \left( \partial^2_{xx} - \kappa \right)} - \kappa \operatorname{R}\left( \partial^2_{xx} + \kappa \right) \txte^{-t \kappa} \left( \txte^{t \partial^2_{xx}} - \txte^{-t \kappa} \right) \right) \kappa \operatorname{R}\left( \partial^2_{xx} - \kappa \right) \\
        &+ \left( \txte^{t \partial^2_{xx}}  - \kappa \operatorname{R}\left( \partial^2_{xx} + \kappa \right) \left( \txte^{t \partial^2_{xx}} - \txte^{-t \kappa} \right) \right) 
        Q \left( \txte^{t \partial^2_{xx}}  - \kappa \operatorname{R}\left( \partial^2_{xx} + \kappa \right) \left( \txte^{t \partial^2_{xx}} - \txte^{-t \kappa} \right) \right) - Q \Bigg) \nonumber
    \end{align}
    holds. For $(n_1,n_2)\in\mathbb{N}\times\mathbb{N}$, we label $p_{n_1,n_2}=\left\langle e_{n_1}, Q e_{n_2} \right\rangle$. For $(n_1,n_2)\in\mathbb{N}\times\mathbb{N}\setminus(0,0)$, the equation \eqref{eq:finite-time_lyap_small_2} implies the definition of $\gamma_{n_1,n_2}:= \left\langle e_{n_1}, V^{\RomanNumeralCaps{1}}_t e_{n_2} \right\rangle$, as described in \eqref{eq:alternative_gammas} and \eqref{eq:alternative_gammas_0}. Lastly, the form
        \begin{align*}
            V_t= \sigma_{\xi}^2 \int_0^t \txte^{s A} \begin{pmatrix}
                \sigma_{\text{R}}^2 Q & \sigma_{\text{R}} Q \\ \sigma_{\text{R}} Q & Q
                \end{pmatrix} \txte^{s A^*} \text{d} s ,
        \end{align*}
        implies the construction of $\gamma_{0,0}$,
    \begin{align*}
        \gamma_{0,0}&=\left\langle \begin{pmatrix}
            e_0 \\ 0
        \end{pmatrix}, V_t \begin{pmatrix}
            e_0 \\ 0
        \end{pmatrix} \right\rangle_{L^2([0,L])\times L^2([0,L])}
        = \left\langle e_0, V^{\RomanNumeralCaps{1}}_t e_0 \right\rangle \\
        &= \sigma_{\text{R}}^2 \sigma_{\xi}^2 \int_0^t \left\langle e_0, \left( \txte^{s \partial^2_{xx}} - \kappa \operatorname{R}\left( \partial^2_{xx} + k \right) \left( \txte^{s \partial^2_{xx}} - \txte^{-s \kappa} \right) \right) Q \left( \txte^{s \partial^2_{xx}} - \kappa \operatorname{R}\left( \partial^2_{xx} + k \right) \left( \txte^{s \partial^2_{xx}} - \txte^{-s \kappa} \right) \right) e_0 \right\rangle \text{d}s\\
        &= \sigma_{\text{R}}^2 \sigma_{\xi}^2 p_{0,0} \int_0^t \left( 1 - \left( 1 - \txte^{-s \kappa} \right) \right) \text{d}s 
        = \sigma_{\text{R}}^2 \sigma_{\xi}^2 p_{0,0} \frac{1-\txte^{-2 t \kappa}}{2 \kappa} .
    \end{align*}
    For these constants $\left\{ \gamma_{n_1,n_2} \right\}_{(n_1,n_2)\in \mathbb{N}\times\mathbb{N}}$, we obtain lower bounds of the probability of turbulence initiation as described in Remark \ref{rmk:other_red} for $q=q(x,t)$, strong solution of \eqref{eq:syst_q}, associated to the choices of $\sigma_{\text{I}}$, $\sigma_{\text{S}}$, $\sigma_{\text{R}}$ and $F$ assumed in \eqref{eq:syst_w_SR2}.

\end{document}